\documentclass[11pt,letterpaper]{amsart}
\usepackage{amssymb}
\usepackage[dvips]{color}
\usepackage{amsmath}
\usepackage{amsthm}
\usepackage{bbm}
\usepackage{amscd,comment,euscript,graphics}
\allowdisplaybreaks[4]

 \newtheorem{thm}{Theorem}[section]
 \newtheorem{cor}[thm]{Corollary}
 \newtheorem{lem}[thm]{Lemma}
 \newtheorem{prop}[thm]{Proposition}
 \theoremstyle{definition}
 \newtheorem{defn}[thm]{Definition}
 \theoremstyle{remark}
 \newtheorem{rem}[thm]{Remark}
 \numberwithin{equation}{section}

 \newcommand{\h}{\mathcal{H}}
 
 \newcommand{\A}{\mathcal{A}}

 \newcommand{\C}{\mathbb{C}}

 \newcommand{\abs}[1]{\left\vert#1\right\vert}

 \newcommand{\norm}[1]{\left\Vert#1\right\Vert}

\begin{document}
\title
{Singular Masas  and Measure-Multiplicity Invariant}

\author{ Kunal Mukherjee }

\address{\hskip-\parindent
 Department of Mathematics \\
 Texas A\&M University \\
 College Station TX 77843--3368, USA}

\email{kunal@imsc.res.in;kunal@neo.tamu.edu}
\begin{abstract}
In this paper we study relations between the
\emph{left-right-measure} and properties of singular masas. Part of
the analysis is mainly concerned with masas for which the
\emph{left-right-measure} is the class of product measure. We
provide examples of Tauer masas in the hyperfinite $\rm{II}_{1}$
factor whose \emph{left-right-measure} is the class of Lebesgue
measure. We show that for each subset $S\subseteq \mathbb{N}$, there
exist uncountably many pairwise non conjugate singular masas in the free
group factors with \emph{Puk\'{a}nszky invariant} $S\cup\{\infty\}$.
\end{abstract}
\thanks{}



\date{30 April 2010}

\dedicatory{}

\commby{}

\maketitle

\section{Introduction and Preliminaries}

Throughout the entire paper, $\mathcal{M}$ will denote a separable
$\rm{II}_{1}$ factor equipped with its faithful normal tracial state $\tau$. This trace
gives rise to a Hilbert norm on $\mathcal{M}$, given by
$\norm{x}_{2}=\tau(x^{*}x)^{\frac{1}{2}}$, $x\in \mathcal{M}$. The
Hilbert space completion of $\mathcal{M}$ with respect to
$\norm{\cdot}_{2}$ is denoted by $L^{2}(\mathcal{M})$. Let
$\mathcal{M}$ act on $L^{2}(\mathcal{M})$ via left multiplication.
Let $A\subset \mathcal{M}$ be a maximal abelian self-adjoint subalgebra $($masa$)$. Dixmier in
\cite{MR0059486} defined the group of \emph{normalizing unitaries} $($or \emph{normalizer}$)$ of
$A$ to be the set
$$N(A)=\left\{u\in \mathcal{U}(\mathcal{M}): uAu^{*}=A\right\},$$
where $\mathcal{U}(\mathcal{M})$ denotes the unitary group of
$\mathcal{M}$. He called \\
$(i)$ $A$ to be \emph{regular} $($also \emph{Cartan}$)$ if $N(A)^{\prime\prime}=\mathcal{M}$,\\
$(ii)$ $A$ to be \emph{semiregular} if $N(A)^{\prime\prime}$ is a
subfactor of $\mathcal{M}$,\\
$(iii)$ $A$ to be \emph{singular} if $N(A)\subset A$.\\
\indent Two masas $A,B$ of $\mathcal{M}$ are said to be
\emph{conjugate}, if there is an automorphism $\theta$ of
$\mathcal{M}$ such that $\theta(A)=B$. If there is an unitary $u\in
\mathcal{M}$ such that $uAu^{*}=B$, then $A$ and $B$ are called
\emph{unitarily} $($\emph{inner}$)$ \emph{conjugate}.
One of the most fundamental problem regarding masas is to decide the conjugacy of two masas.
The most successful invariant so far in this regard is the \emph{Puk\'{a}nszky invariant} \cite{MR0113154}.
Nevertheless, it is not a complete invariant.\\
\indent The \emph{measure-multiplicity invariant} of masas in $\rm{II}_{1}$
factors was studied in \cite{MR2261688,MR11932708,MR1940356}. It was
used in \cite{MR2261688} to distinguish two masas with the same
\emph{Puk\'{a}nszky invariant}. It is a stronger invariant than the
\emph{Puk\'{a}nszky invariant}. It has two main
components, a measure class $($\emph{left-right-measure}$)$ and a
multiplicity function, which together encode the structure of the
standard Hilbert space as an associated bimodule. In this paper, we
study analytical relations between the \emph{left-right-measure} and
properties of singular masas. We focus on the following question: To
what extent does the standard Hilbert space as a natural bimodule remember
properties of the masa. In \cite{MR11932708}, we established that
\emph{left-right-measure} has all information to measure the size of $N(A)$ $($see Thm. 5.5 \cite{MR11932708}$)$.\\
\indent In this paper, we consider different kinds of singular masas.
We introduce a condition on masas which forces vigorous mixing properties. Such masas are
automatically strongly mixing \cite{MR2465603} $($for a proof see Thm. 9.2 \cite{MR999999999}$)$ and consequently singular. We show that if $A$ is such a masa
in $\mathcal{M}$ with singleton multiplicity, then the Hilbert space $L^{2}(\mathcal{M})\ominus L^{2}(A)$
as a natural $A,A$-bimodule is a direct sum of copies of $L^{2}(A)\otimes L^{2}(A)$, i.e., we show that its \emph{left-right-measure} is the class of product measure.
We also present a converse to the foresaid statement. The arguments required to
prove this statement show that, if $B\subset A$ is diffuse, then
$L^{2}(\mathcal{M})\ominus L^{2}(B^{\prime}\cap \mathcal{M})$ as a $B,B$-bimodule
is a direct sum of copies of submodules of $L^{2}(B)\otimes L^{2}(B)$. There is an abundance of such
masas in the hyperfinite $\rm{II}_{1}$ factor, but there are fewer examples of such masas, if in addition we demand that such a masa
has a bicyclic vector. We also study the \emph{left-right-measure} of $\Gamma$ and
non-$\Gamma$ singular masas. In particular, we show that under certain extra assumption
on central sequences, the presence of central sequences in a masa can be related to
\emph{rigid measures}. Examples of such masas come from Ergodic theory.\\
\indent The following question asked by Banach is a long standing open problem in Ergodic theory.
Does there exist a simple measure preserving $($m.p.$)$ automorphism with pure Lebesgue spectrum?
It is implicit in the above question that one is asking about an action of $\mathbb{Z}$. Translated
to operator algebras this means $($see for instance \cite{MR1940356}$)$, whether there is a way to construct the hyperfinite $\rm{II}_{1}$ factor as $L^{\infty}(X,\mu)\rtimes \mathbb{Z}$, where $L(\mathbb{Z})$ is a \emph{simple masa} whose \emph{left-right-measure} is the class of product measure.
The term `simple' of course means simple multiplicity or equivalently, the existence of a bicyclic vector of $A$.\\
\indent We provide an example of such a Tauer masa in the hyperfinite $\rm{II}_{1}$
factor. All Tauer masas are simple \cite{MR2253595}. We do not know if this example arises from an action of
integers or any other group action. But quite surprisingly Banach's problem
has an easy and affirmative answer if we change the group.
Using the methods developed in \cite{MR2261688}, we show that for
each subset $S\subseteq \mathbb{N}$ $($could be empty$)$, there are
uncountably many pairwise non conjugate singular masas in the free group
factors with \emph{Puk\'{a}nszky invariant} $S\cup \{\infty\}$.\\
\indent This paper is organized as follows. We provide the background
material in this section itself. In \S2, we study masas for which the \emph{left-right-measure} is the class of product measure. \S3 is devoted to Tauer masas. \S4 contains partial
results regarding the \emph{left-right-measure} of $\Gamma$ and
non-$\Gamma$ masas. In \S5, we exhibit examples of singular masas in
free group factors.\\
\indent Let $J$ denote the Tomita's modular operator on
$L^{2}(\mathcal{M})$, obtained by extending the densely defined
map $J:\mathcal{M}\mapsto \mathcal{M}$ by $Jx=x^{*}$. The image of a $L^{2}$ vector $\zeta$ under $J$ will be denoted by $\zeta^{*}$. Let $e_{A}:L^{2}(\mathcal{M})\mapsto
L^{2}(A)$ be the Jones projection associated to $A$. Denote
$\A=(A\cup JAJ)^{\prime\prime}$. It is known that $e_{A}\in
\A$ $($Thm. 3.1 \cite{MR815434}$)$. Let $\mathbb{E}_{A}$ denote the unique, normal, trace
preserving conditional expectation from $\mathcal{M}$ on to $A$. The
conditional expectation $\mathbb{E}_{A}$ and the trace extends to
$L^{1}(\mathcal{M})$ in a continuous fashion $($see \S B.5
\cite{MR999996}$)$. With abuse of notation, we will write
$e_{A}(\zeta)=\mathbb{E}_{A}(\zeta)$ for $L^{1}$ and $L^{2}$
vectors. Similarly, we will use the same symbol $\tau$ to denote its
extension. This will be clear from the context and will cause no
confusion. This work relies on direct integrals. For standard
results on direct integrals we refer the reader to \cite{MR641217}.
Throughout the entire paper $\mathbb{N}_{\infty}$ will denote the
set $\mathbb{N}\cup \{\infty\}$. For a set $X$, we will write $\Delta(X)$ to denote the diagonal of $X\times X$.

\begin{defn}
Given a type $\rm{I}$ von Neumann algebra $B$, we shall write
Type($B$) for the set of all those $n\in\mathbb{N}_\infty$ such that
$B$ has a nonzero component of type $\rm{I}_{n}$.
\end{defn}

\begin{defn}\cite{MR0113154}\label{Puk_invariant}
The \emph{Puk\'{a}nszky invariant} of $A\subset \mathcal{M}$,
denoted by $Puk(A)$ (or $Puk_{\mathcal{M}}(A)$ when the containing
factor is ambiguous) is Type($\A^{\prime}(1 - e_{A}))$.
\end{defn}

\begin{defn}\cite{MR2261688,MR11932708,MR1940356}\label{mminv}
The \emph{measure-multiplicity invariant} of $A\subset \mathcal{M}$,
denoted by $m.m(A)$, is the equivalence class of quadruples
$(X,\lambda_{X},[\eta_{\mid\Delta(X)^{c}}],m_{\mid\Delta(X)^{c}})$
under the equivalence relation
$\sim_{m.m}$, where,\\
$(i)$ $X$ is a compact Hausdorff space such that $C(X)$ is an unital, norm separable, \emph{w.o.t} dense subalgebra of $A$,\\
$(ii)$ $\lambda_{X}$ is the Borel probability measure obtained by restricting the trace $\tau$ on $C(X)$,\\
and,\\
$(iii)$ $\eta_{\mid\Delta(X)^{c}}$ is the measure on $X\times X$ concentrated on $\Delta(X)^{c}$, and\\
$(iv)$ $m_{\mid\Delta(X)^{c}}$ is the multiplicity function restricted to $\Delta(X)^{c}$,\\
obtained from the direct integral decomposition of
$L^{2}(\mathcal{M})\ominus L^{2}(A)$, so that $\A(1-e_{A})$ is the
algebra
of diagonalizable operators with respect to this decomposition, the equivalence $\sim_{m.m}$ being,
$(X,\lambda_{X},[\eta_{\mid\Delta(X)^{c}}],m_{\mid\Delta(X)^{c}})\sim_{m.m}(Y,\lambda_{Y},[\eta_{\mid\Delta(Y)^{c}}],m_{\mid\Delta(Y)^{c}})$,
if and only if, there exists a Borel isomorphism $F:X\mapsto Y$ such
that,
\begin{align}
\nonumber & F_{*}\lambda_{X}=\lambda_{Y},\\
\nonumber & (F\times F)_{*}[\eta_{\mid\Delta(X)^{c}}]=[\eta_{\mid\Delta(Y)^{c}}],\\
\nonumber & m_{\mid\Delta(X)^{c}}\circ(F\times
F)^{-1}=m_{\mid\Delta(Y)^{c}}, \eta_{\mid\Delta(Y)^{c}} \text{ a.e.}
\end{align}
\end{defn}
\indent It is easy to see that the \emph{Puk\'{a}nszky invariant} of
$A\subset \mathcal{M}$ is the set of essential values of the
multiplicity function in Defn. \ref{mminv}. The measure class
$[\eta_{\mid\Delta(X)^{c}}]$ in Defn. \ref{mminv} is said to be the
\emph{left-right-measure} of $A$. Both $m.m(\cdot)$ and $Puk(\cdot)$
are invariants of the masa under automorphisms of the factor
$\mathcal{M}$. Given a pair of masas, the \emph{mixed Puk\'{a}nszky
invariant} was introduced in \cite{MR999998}. Analogous to the
\emph{mixed Puk\'{a}nszky invariant}, one can define the
\emph{joint-measure-multiplicity invariant} for pair of masas $A$ and $B$, by considering the direct integral decomposition of
$L^{2}(\mathcal{M})$ with respect to $(A\cup JBJ)^{\prime\prime}$. For
details check Ch. \rm{V} \cite{MR1000012}. Such invariants play a role in
questions concerning unitary conjugacy of masas.

In some cases, it is necessary to have a direct integral
decomposition of $L^{2}(\mathcal{M})$. These are situations when one
considers tensors of masas. In these cases, the information on the
diagonal $\Delta(X)$ is to be supplied. Following \S2.3 of \cite{MR11932708}
note that
\begin{align}
\nonumber L^{2}(\mathcal{M})\cong\int_{X\times
X}^{\oplus}\h_{t,s}d(\eta_{\mid\Delta(X)^{c}}+\tilde{\Delta}_{*}\lambda_{X})(t,s),
\end{align}
where $\tilde{\Delta}:X\mapsto X\times X$ by $t\mapsto (t,t)$,
$\h_{t,t}=\C$ for $\lambda_{X}$ almost all $t$, $\h_{t,s}$ depends
on the \emph{Puk\'{a}nszky invariant} and $\A$ is diagonalizable
with respect to this decomposition. In such cases, we will also call
$[\eta_{\mid\Delta(X)^{c}}+\tilde{\Delta}_{*}\lambda_{X}]$ to be the
\emph{left-right-measure} of $A$. It is to be understood that, when
we consider direct integrals or make statements about
diagonalizability, we need to complete the measures under
consideration.

For a masa $A\subset \mathcal{M}$, fix a compact Hausdorff space $X$
such that $C(X)\subset A$ is an unital, norm separable and
\emph{w.o.t} dense $C^{*}$ subalgebra. Let $\lambda$ denote the
tracial measure on $X$. For $\zeta_{1},\zeta_{2}\in
L^{2}(\mathcal{M})$, let $\kappa_{\zeta_{1},\zeta_{2}}:C(X)\otimes
C(X)\mapsto \C$ be the linear functional defined by,
\begin{align}
\nonumber \kappa_{\zeta_{1},\zeta_{2}}(a\otimes b)=\langle
a\zeta_{1} b,\zeta_{2}\rangle,\text{ } a,b\in C(X).
\end{align}
Then $\kappa_{\zeta_{1},\zeta_{2}}$ induces an unique complex Radon
measure $\eta_{\zeta_{1},\zeta_{2}}$ on $X\times X$ given by,
\begin{align}\label{measure_from_kappa}
\kappa_{\zeta_{1},\zeta_{2}}(a\otimes b)=\int_{X\times
X}a(t)b(s)d\eta_{\zeta_{1},\zeta_{2}}(t,s),
\end{align}
and
$\norm{\eta_{\zeta_{1},\zeta_{2}}}_{t.v}=\norm{\kappa_{\zeta_{1},\zeta_{2}}}$,
where $\norm{\cdot}_{t.v}$ denotes the total variation norm of
measures.

We will write $\eta_{\zeta,\zeta}=\eta_{\zeta}$. Note that
$\eta_{\zeta}$ is a positive measure for all $\zeta\in
L^{2}(\mathcal{M})$. It is easy to see that the following
polarization type identity holds:
\begin{align}\label{polarize}
4\eta_{\zeta_{1},\zeta_{2}}=\left(\eta_{\zeta_{1}+\zeta_{2}}-\eta_{\zeta_{1}-\zeta_{2}}\right)+i\left(\eta_{\zeta_{1}+i\zeta_{2}}-\eta_{\zeta_{1}-i\zeta_{2}}
\right).
\end{align}
Note that the decomposition of $\eta_{\zeta_{1},\zeta_{2}}$ in Eq.
\eqref{polarize} need not be its Hahn decomposition in general, but
\begin{align}\label{polarization1}
4\abs{\eta_{\zeta_{1},\zeta_{2}}}&\leq\left(\eta_{\zeta_{1}+\zeta_{2}}+\eta_{\zeta_{1}-\zeta_{2}}\right)+\left(\eta_{\zeta_{1}+i\zeta_{2}}+\eta_{\zeta_{1}-i\zeta_{2}}\right)
=4(\eta_{\zeta_{1}}+\eta_{\zeta_{2}}).
\end{align}
So
\begin{align}\label{abscont11}
\abs{\eta_{\zeta_{1},\zeta_{2}}}\leq
\eta_{\zeta_{1}}+\eta_{\zeta_{2}}.
\end{align}
\indent To understand the relation between properties of masas and
their \emph{left-right-measure}, disintegration of measures will be
used, for which we refer the reader to \S2 of \cite{MR1484954}. Let $T$ be
a measurable map from $(X,\sigma_{X})$ to $(Y,\sigma_{Y})$, where
$\sigma_{X},\sigma_{Y}$ are $\sigma$-algebras of subsets of $X,Y$
respectively. Let $\beta$ be a $\sigma$-finite measure on
$\sigma_{X}$ and $\mu$ a $\sigma$-finite measure on $\sigma_{Y}$.
Here $\beta$ is the measure to be disintegrated and $\mu$ is often
the push forward measure $T_{*}\beta$, although other possibilities
for $\mu$ is allowed.
\begin{defn}\cite{MR1484954}\label{definition_of_disintegration}
We say that $\beta$ has a disintegration $\{\beta^{t}\}_{t\in Y}$
with
respect to $T$ and $\mu$ or a $(T,\mu)$-disintegration if:\\
$(i)$ $\beta^{t}$ is a $\sigma$-finite measure on $\sigma_{X}$
concentrated on $\{T=t\}$ $($or $T^{-1}\{t\})$, i.e.,
$\beta^{t}(\{T\neq t\})=0$, for
$\mu$-almost all $t$,\\
and, for each nonnegative $\sigma_{X}$-measurable function $f$ on $X$:\\
$(ii)$ $t\mapsto \beta^{t}(f)$ is $\sigma_{Y}$-measurable.\\
$(iii)$
$\beta(f)=\mu^{t}(\beta^{t}(f))\overset{\text{defn}}=\int_{Y}\beta^{t}(f)d\mu(t)$.
\end{defn}
If $\beta$ in Defn. \ref{definition_of_disintegration} is a complex
measure, then the disintegration of $\beta$ is obtained by
decomposing it into a linear combination of four positive measures,
using the Hahn decomposition of its real and imaginary parts.\\
\noindent \textbf{Notation:} The disintegrated measures are usually
written with a subscript $t\mapsto \beta_{t}$ in the literature. But
in this paper, we will use the superscript notation $t\mapsto
\beta^{t}$ to denote them. The $(\pi_{1},\lambda)$-disintegration of
measures on $X\times X$
will be indexed by the variable $t$  and the
$(\pi_{2},\lambda)$-disintegration will be indexed by the variable
$s$, where $\pi_{i}:X\times X\mapsto X$, $i=1,2$, are the coordinate
projections. We will only consider the
$(\pi_{i},\lambda)$-disintegrations of the measures $\eta_{\zeta},
\eta_{\zeta_{1},\zeta_{2}}$ defined in Eq.
\eqref{measure_from_kappa}. These disintegrations exist from Thm.
3.2 \cite{MR11932708} $($also see Thm. 1 \cite{MR1484954}$)$. The measure
$\eta_{\zeta}^{t}$ is concentrated on $\{t\}\times X$ and the
measure $\eta_{\zeta}^{s}$ is concentrated on $X\times \{s\}$ for
$\lambda$ almost all $t,s$ respectively. We will denote by
$\tilde{\eta}_{\zeta}^{t}$ the restriction of the measure
$\eta_{\zeta}^{t}$ on $\{t\}\times X$. Similarly define
$\tilde{\eta}_{\zeta}^{s}$. Thus,
$\tilde{\eta}_{\zeta}^{t},\tilde{\eta}_{\zeta}^{s}$ can be regarded
as measures on $X$.\\
\indent The \emph{left-right-measure} $[\eta]$ of $A$ has the
following property. If $\theta:X\times X\mapsto X\times X$ is the
flip map i.e., $\theta(t,s)=(s,t)$, then
$\theta_{*}\eta\ll\eta\ll\theta_{*}\eta$ $($see Lemma 2.9
\cite{MR11932708}$)$. In fact, it is possible to obtain a choice of
$\eta$ for which $\theta_{*}\eta=\eta$. So in most of the analysis,
we will only state or prove results with respect to the
$(\pi_{1},\lambda)$-disintegration. An analogous statement with
respect to the $(\pi_{2},\lambda)$-disintegration is also possible.
We will only work with finite or probability measures.

\section{The Product Class}

It is not always easy to describe the properties of a singular masa
based on its \emph{left-right-measure}. However, we can write interesting
properties of masas when the \emph{left-right-measure} is the class
of product measure. Such masas are singular Thm. 5.5 \cite{MR11932708}.
Examples of such masas are easy to obtain in many situations and many
known masas, for example, the single generator masas in the free group factors, the masas that arise out of Bernoulli
shift actions of countable discrete abelian groups belong to this
class. In this section, we shall give analytical conditions for the
\emph{left-right-measure} of a masa to be the class of product measure.\\
\indent Let $\lambda$ denote the Lebesgue measure on $[0,1]$ so that
$A\cong L^{\infty}([0,1],\lambda)$. Then $\lambda$ is the tracial
measure. Let $[\eta]$ denote the \emph{left-right-measure} of $A$.
We assume that $\eta$ is a probability measure on $[0,1]\times
[0,1]$ and $\eta(\Delta([0,1]))=0$.\\
\indent `The \emph{left-right-measure} of any masa in the
free group factors contains a part of $\lambda\otimes \lambda$ as
a summand'. This statement of Voiculescu is one of
the most important theorem in the subject $($Cor. 7.6 \cite{MR1371236} applied to a system of free semicirulars does the job.$)$ This is the precise
reason for the absence of Cartan subalgebras in the free
group factors. In many cases, the \emph{left-right-measures} are
difficult to calculate. So we need conditions in terms of operators that
characterize the Lebesgue class. The following is the main result of this section.
It will be proved later in this section.

\begin{thm}\label{sum_cond}
The left-right-measure of a masa $A\subset \mathcal{M}$ is the class of product measure, if there exists a set $S\subset
\mathcal{M}$ such that $\mathbb{E}_{A}(x)=0$ for all $x\in S$ and\\
$(i)$  the linear span of $S$ \text{ is dense in }$L^{2}(\mathcal{M})\ominus L^{2}(A)$,\\
$(ii)$ there is an orthonormal basis
$\{v_{n}\}_{n=1}^{\infty}\subset A$ of $L^{2}(A)$ such that
\begin{align}
\nonumber \sum_{n=1}^{\infty}\norm{\mathbb{E}_{A}(x
v_{n}x^{*})}_{2}^{2}<\infty \text{ for all }x\in S,
\end{align}
$(iii)$ there is a nonzero vector $\zeta\in
L^{2}(\mathcal{M})\ominus L^{2}(A)$ such that $\mathbb{E}_{A}(\zeta
u^{n}\zeta^{*})=0$ for all $n\neq 0$, where $u$ is a Haar unitary
generator of $A$.
\end{thm}

We do not know whether the conditions in Thm. \ref{sum_cond} are necessary for the same conclusion to hold. In Thm. \ref{measure_to_sum_all_puk}, we provide an analogous condition which is necessary for the
\emph{left-right-measure} to be of the product class. In general, it is of interest to
know whether there exist masas for which $\eta\ll \lambda\otimes \lambda$ but
$[\eta]\neq [\lambda\otimes \lambda]$. Note that the
sum in Thm. \ref{sum_cond} is independent of the choice of the
orthonormal basis. This just follows by expanding elements of one
orthonormal basis with respect to another. Hence by making similar arguments, $(iii)$ in Thm. \ref{sum_cond} holds for any Haar unitary generator of $A$, if it holds for one Haar unitary generator. Conditions $(i)$ and
$(ii)$ in Thm. \ref{sum_cond} forces that $\eta\ll \lambda\otimes
\lambda$. To assure $\lambda\otimes \lambda\ll\eta$ we need
condition $(iii)$ $($Cor. \ref{leb_confirm}$)$. As it will become clear, these conditions are analogous to knowing a measure from the information of its Fourier coefficients. Condition $(iii)$ is an analogue of the fact that the Fourier coefficients of $\lambda$ are $0$ except for the zeroth coefficient. In this sense the operators $\mathbb{E}_{A}(xv_{n}x^{*})$ can be thought of as the `Fourier coefficients of the bimodule $\overline{AxA}^{\norm{\cdot}_{2}}$'.\\
\indent In order to motivate the conditions in Thm. \ref{sum_cond}, we cite some examples. Conditions $(i)$ and $(ii)$ first appeared in the study of radial masas in the free group factors $($see Thm. 3.1 \cite{MR1932708}$)$. In this case, the natural choice of the orthonormal basis is $v_{n}=\frac{\chi_{n}}{\norm{\chi_{n}}_{2}}$, $n\geq 0$, where $\chi_{n}=\sum_{w:\abs{w}=n}w$. The set $S$ consists of $w-\mathbb{E}_{A}(w)$, $1\neq w\in \mathbb{F}_{k}$, $k\geq 2$. Single generator masas in the free group factors clearly satisfies all the three conditions. In fact, for all masas exhibited in \S6 there is a nonzero vector for which $(iii)$ in Thm. \ref{sum_cond} is satisfied. However, in those examples $(i)$ won't be satisfied. \\
\indent The second class of examples comes from inclusion of groups. Suppose $\Gamma$ is a countable discrete icc group with $\Lambda<\Gamma$, such that $L(\Lambda)\subset L(\Gamma)$ is a singular masa. Assume further that $\Lambda$ is malnormal in $\Gamma$. It is obvious that $(i)$, $(ii)$ and $(iii)$ hold $($Thm. 9.5, Cor. 9.8 \cite{MR999999999}$)$.
Similar class of examples consists of freely complemented masas.\\
\indent Next class of examples comes from ergodic theory. Consider a countable infinite abelian group $\Gamma$ acting on $\prod_{\Gamma}(X,\mu)$ by Bernoulli shift, where $(X,\mu)$ is a probability space with at least two points. Then $L(\Gamma)\subset L^{\infty}(\prod_{\Gamma}(X,\mu))\rtimes \Gamma$ is a singular masa whose \emph{left-right-measure} is the class of product Haar measure on $\widehat{\Gamma}\times \widehat{\Gamma}$ $($Prop. 3.1 \cite{MR1940356}$)$. For any function
$f$ on the $\gamma$-th copy, $\gamma\in \Gamma$, such that $\mu(f)=0$ one has
\begin{align}
\nonumber \mathbb{E}_{L(\Gamma)}(fu_{\gamma^{\prime}}f^{*})=\tau(f\alpha_{\gamma^{\prime}}(f^{*}))u_{\gamma^{\prime}}, \gamma^{\prime}\in \Gamma,
\end{align}
where $u_{\gamma^{\prime}}$ are the canonical group unitaries in the crossed product, $\tau$ is the tracial state and $\alpha_{\gamma^{\prime}}$ denotes the automorphism corresponding to $\gamma^{\prime}$. It follows that $\sum_{\gamma^{\prime}\in \Gamma}\norm{\mathbb{E}_{L(\Gamma)}(fu_{\gamma^{\prime}}f^{*})}_{2}^{2}$ is finite. It is now clear that $(i)$ and $(ii)$ of Thm. \ref{sum_cond} hold. Since the maximal spectral type of the Bernoulli action is the normalized Haar measure on $\widehat{\Gamma}$, any function $f\in L^{2}(\prod_{\Gamma}(X,\mu))$ for which the maximal spectral type is attained,  satisfies $\mathbb{E}_{L(\Gamma)}(fu_{\gamma}f^{*})=0$ for all $\gamma\neq 1$. The last statement is equivalent to $(iii)$.

In order to prove Thm. \ref{sum_cond},  we need to prove some auxiliary lemmas.

\begin{lem}\label{identify_disintegrated_measure}
Let $\zeta_{1},\zeta_{2}\in L^{2}(\mathcal{M})$ be such that
$\mathbb{E}_{A}(\zeta_{1})=0=\mathbb{E}_{A}(\zeta_{2})$. Let
$\eta_{\zeta_{1},\zeta_{2}}$ denote the Borel measure on
$[0,1]\times
[0,1]$ defined in Eq. \eqref{measure_from_kappa}.\\
$1^{\circ}$. Then $\eta_{\zeta_{1},\zeta_{2}}$ admits
$(\pi_{i},\lambda)$-disintegrations $[0,1]\ni
t\mapsto\eta_{\zeta_{1},\zeta_{2}}^{t}$ and $[0,1]\ni
s\mapsto\eta_{\zeta_{1},\zeta_{2}}^{s}$, where $\pi_{i}$, $i=1,2$,
denotes the coordinate projections. Moreover,
\begin{align}
\nonumber \eta_{\zeta_{1},\zeta_{2}}^{t}([0,1]\times
[0,1])=\mathbb{E}_{A}(\zeta_{1}\zeta_{2}^{*})(t), \text{ }\lambda
\text{ a.e.}
\end{align}
$2^{\circ}$. Let $f\in C[0,1]$. Then the functions $[0,1]\ni
t\mapsto \eta_{\zeta_{1},\zeta_{2}}^{t}(1\otimes f), [0,1]\ni
s\mapsto \eta_{\zeta_{1},\zeta_{2}}^{s}(f\otimes 1)$ are in
$L^{1}([0,1],\lambda)$. \\
If $\zeta_{i}\in \mathcal{M}$ for $i=1,2$, then $[0,1]\ni t\mapsto
\eta_{\zeta_{1},\zeta_{2}}^{t}(1\otimes f), [0,1]\ni s\mapsto
\eta_{\zeta_{1},\zeta_{2}}^{s}(f\otimes 1)$ are in
$L^{\infty}([0,1],\lambda)$.\\
$3^{\circ}$. Let $b,w\in C[0,1]$. If $\mathbb{E}_{A}(\zeta_{1}
w\zeta_{2}^{*})\in L^{2}(A)$, then
\begin{align}
\nonumber \norm{\mathbb{E}_{A}(b\zeta_{1} w\zeta_{2}
^{*})}_{2}^{2}=\int_{0}^{1}\abs{b(t)}^{2}\abs{{\eta_{\zeta_{1},\zeta_{2}}^{t}(1\otimes
w)}}^{2}d\lambda(t).
\end{align}
\end{lem}

\begin{proof}
$1^{\circ}$. That $\eta_{\zeta_{1},\zeta_{2}}$ admits the stated
disintegrations follows from Eq. \eqref{polarize}, Lemma 5.7
\cite{MR2261688} and Lemma 2.9 \cite{MR11932708}. The next
statement in $1^{\circ}$ follows from an
argument similar to the proof of Lemma 6.1 \cite{MR11932708}.\\
$2^{\circ}$. From Eq. \eqref{polarization1},
$\abs{\eta_{\zeta_{1},\zeta_{2}}}$ admits
$(\pi_{i},\lambda)$-disintegrations. Use Hahn decomposition of
measures and Lemma 3.6 \cite{MR11932708} to see that
$\abs{\eta_{\zeta_{1},\zeta_{2}}}^{t}=\abs{\eta_{\zeta_{1},\zeta_{2}}^{t}}$
for $\lambda$ almost all $t$. The function $t\mapsto
\eta_{\zeta_{1},\zeta_{2}}^{t}(1\otimes f)$ is clearly measurable
from Defn. \ref{definition_of_disintegration}, and from Eq.
\eqref{abscont11} we have
\begin{align}
\nonumber &\int_{0}^{1}\abs{\eta_{\zeta_{1},\zeta_{2}}^{t}(1\otimes f)}d\lambda(t)\leq \norm{f}\int_{0}^{1}\abs{\eta_{\zeta_{1},\zeta_{2}}^{t}}([0,1]\times [0,1])d\lambda(t)\\
\nonumber \leq& \norm{f}\left(\int_{0}^{1}\eta_{\zeta_{1}}^{t}([0,1]\times[0,1])d\lambda(t)+\int_{0}^{1}\eta_{\zeta_{2}}^{t}([0,1]\times[0,1])d\lambda(t)\right)\\
\nonumber
=&\norm{f}\left(\norm{\mathbb{E}_{A}(\zeta_{1}\zeta_{1}^{*})}_{1}+\norm{\mathbb{E}_{A}(\zeta_{2}\zeta_{2}^{*})}_{1}\right)<\infty.
\end{align}
\indent When $\zeta_{i}\in \mathcal{M}$ a similar argument shows that the
stated
functions are in $L^{\infty}([0,1],\lambda)$.\\
$3^{\circ}$. Since
\begin{align}
\nonumber\infty >\underset{a\in C[0,1], \norm{a}_{2}\leq 1}\sup\abs{\int_{0}^{1}a(t)b(t)\mathbb{E}_{A}(\zeta_{1} w\zeta_{2}^{*})(t)d\lambda(t)}&=\underset{a\in C[0,1], \norm{a}_{2}\leq 1}\sup\abs{\tau(ab\mathbb{E}_{A}(\zeta_{1}w\zeta_{2}^{*}))}\\
\nonumber&=\underset{a\in C[0,1], \norm{a}_{2}\leq 1}\sup \abs{\tau(ab\zeta_{1}w\zeta_{2}^{*})}\\
\nonumber&=\underset{\norm{a}_{2}\leq 1}{\underset{a\in C[0,1]}\sup}
\abs{\int_{0}^{1}a(t)b(t)\eta_{\zeta_{1},\zeta_{2}}^{t}(1\otimes
w)d\lambda(t)}
\end{align}
and $t\overset{g}\mapsto b(t)\eta_{\zeta_{1},\zeta_{2}}^{t}(1\otimes
w)$ is in $L^{1}(\lambda)$, so $g$ is in $L^{2}(\lambda)$ and
\begin{align}
\nonumber
\norm{\mathbb{E}_{A}(b\zeta_{1}w\zeta_{2}^{*})}_{2}^{2}=\int_{0}^{1}\abs{b(t)}^{2}\abs{\eta_{\zeta_{1},\zeta_{2}}^{t}(1\otimes
w)}^{2}d\lambda(t).
\end{align}
\end{proof}

Let $w:=\{w_{n}\}_{n=1}^{\infty}\subset C[0,1]$ be an orthonormal
basis of $L^{2}(A)$.

\begin{prop}\label{finite_at_fibre_both}
Let $x_{i}\in \mathcal{M}$ for $i=1,2$, be such that
$\mathbb{E}_{A}(x_{i})=0$. Let us suppose that
\begin{align}
\nonumber\sum_{n=1}^{\infty}\norm{\mathbb{E}_{A}(x_{1}
w_{n}x_{2}^{*})}_{2}^{2}<\infty.
\end{align}
If $w^{\prime}:=\{w_{n}^{\prime}\}_{n=1}^{\infty}$ be an orthonormal
sequence in $L^{2}(A)$ with $w_{n}^{\prime}\in C[0,1]$ for all $n$,
then there is a set $F(w,w^{\prime})\subset [0,1]$ which depends on
$w,w^{\prime}$ such that $\lambda(F(w,w^{\prime}))=0$ and for all
$t\in F(w,w^{\prime})^{c}$,
\begin{align}
\nonumber\sum_{n=1}^{\infty}\abs{\eta_{x_{1},x_{2}}^{t}(1\otimes
w^{\prime}_{n})}^{2}\leq\sum_{n=1}^{\infty}\abs{\eta_{x_{1},x_{2}}^{t}(1\otimes
w_{n})}^{2}<\infty.
\end{align}
\end{prop}

\begin{proof}
Note that the hypothesis implies that for any $a\in C[0,1]$,
$$\sum_{n=1}^{\infty}\norm{\mathbb{E}_{A}(ax_{1}
w_{n}x_{2}^{*})}_{2}^{2}<\infty$$ and this sum is independent of the
choice of the orthonormal basis. Therefore, for all $a\in C[0,1]$,
\begin{align}
\nonumber
\sum_{n=1}^{\infty}\norm{\mathbb{E}_{A}(ax_{1} w^{\prime}_{n}x_{2}^{*})}_{2}^{2}\leq\sum_{n=1}^{\infty}\norm{\mathbb{E}_{A}(ax_{1} w_{n}x_{2}^{*})}_{2}^{2}.
\end{align}
\indent Let $r\in A$ be a nonzero projection. Identify $r$ with a
measurable subset $E_{r}$ of $[0,1]$. We can assume $E_{r}$ is a
Borel set. We claim that
\begin{align}\label{proj_inte}
\int_{E_{r}}\sum_{n=1}^{\infty}\abs{\eta_{x_{1},x_{2}}^{t}(1\otimes
w_{n}^{\prime})}^{2}d\lambda(t)\leq
\int_{E_{r}}\sum_{n=1}^{\infty}\abs{\eta_{x_{1},x_{2}}^{t}(1\otimes
w_{n})}^{2}d\lambda(t).
\end{align}
If the claim is true, then by standard measure theory arguments we
are done.\\
\indent First assume $E_{r}$ is a compact set. Choose a sequence of
continuous functions $f_{l}$ such that $0\leq f_{l}\leq 1$ and
$f_{l}\downarrow \chi_{E_{r}}$ pointwise as $l\rightarrow \infty$.
Therefore by Lemma \ref{identify_disintegrated_measure} and monotone
convergence theorem, for all $l$ we have,
\begin{align}
\nonumber\int_{0}^{1}f^{2}_{l}(t)\sum_{n=1}^{\infty}\abs{\eta_{x_{1},x_{2}}^{t}(1\otimes w_{n}^{\prime})}^{2}d\lambda(t)&=\sum_{n=1}^{\infty}\norm{\mathbb{E}_{A}(f_{l}x_{1}w^{\prime}_{n}x_{2}^{*})}_{2}^{2}\\
\nonumber &\leq\sum_{n=1}^{\infty}\norm{\mathbb{E}_{A}(f_{l}x_{1}
w_{n}x_{2}^{*})}_{2}^{2}=\int_{0}^{1}f^{2}_{l}(t)\sum_{n=1}^{\infty}\abs{\eta_{x_{1},x_{2}}^{t}(1\otimes
w_{n})}^{2}d\lambda(t).
\end{align}
\indent Passing to limits, we see that Eq. \eqref{proj_inte} is true
whenever $E_{r}$ is compact. Now use regularity of $\lambda$ to see
that Eq. \eqref{proj_inte} is true for all Borel sets of positive
measure.
\end{proof}

Let $X=\prod_{n=1}^{\infty}C[0,1]$. Equip $X$ with the product topology. Then $X$ is known to be separable and metrizable. Every $\underline{f}\in X$ is a infinite tuple $\underline{f}=(f_{1},f_{2},\cdots)$. Also for a
sequence $\underline{f}^{(n)}\in X$,
$\underline{f}^{(n)}\rightarrow \underline{f}$ as
$n\rightarrow \infty$ implies that $f^{(n)}_{k}\rightarrow f_{k}$ in
$\norm{\cdot}_{\infty}$ for all $k\in \mathbb{N}$.\\
\indent Let $\mathcal{O}=\left\{\underline{f}\in X:
\{f_{k}\}_{k=1}^{\infty}\text{ is an orthonormal sequence in }
L^{2}([0,1],\lambda) \right\}$. Then $\mathcal{O}\subset X$ is a
closed set. Note that $\mathcal{O}$ is separable in the product topology.

\begin{prop}\label{product_measure_at_each_step}
Let $x\in \mathcal{M}$ be such that $\mathbb{E}_{A}(x)=0$. Let us
suppose that
\begin{align}
\nonumber\sum_{k=1}^{\infty}\norm{\mathbb{E}_{A}(x
w_{k}x^{*})}_{2}^{2}<\infty.
\end{align}
Then $\eta_{x}\ll \lambda\otimes \lambda$.
\end{prop}
\begin{proof}
Let $\{\underline{w}^{(m)}\}_{m=1}^{\infty}\subset \mathcal{O}$
be any countable dense set. From Prop. \ref{finite_at_fibre_both}
and Lemma \ref{identify_disintegrated_measure}, it follows that
there is a set $F\subset [0,1]$ with $\lambda(F)=0$ such that for
$t\in F^{c}$, $\eta_{x}^{t}$ is a finite measure and
\begin{align}
\nonumber\sum_{k=1}^{\infty}\abs{{\eta^{t}_{x}}(1\otimes
w^{(m)}_{k})}^{2}\leq \sum_{k=1}^{\infty}\abs{\eta_{x}^{t}(1\otimes
w_{k})}^{2}<\infty
\end{align}
for all $m\in \mathbb{N}$. Let
$\underline{v}=\{v_{k}\}_{k=1}^{\infty}\in \mathcal{O}$. There
exists a subsequence $\{\underline{w}^{(m_{j})}\}_{j=1}^{\infty}$
such that $\underline{w}^{(m_{j})}\rightarrow\underline{v}$ as
$j\rightarrow \infty$. Therefore for $t\in F^{c}$,
\begin{align}
\nonumber\sum_{k=1}^{\infty}\abs{{\eta^{t}_{x}}(1\otimes v_{k})}^{2}&=\sum_{k=1}^{\infty}\underset{j}\lim\abs{{\eta^{t}_{x}}(1\otimes w_{k}^{(m_{j})})}^{2}\text{ (by Dominated convergence)}\\
\nonumber&=\sum_{k=1}^{\infty}\underset{j}\liminf\abs{{\eta^{t}_{x}}(1\otimes w_{k}^{(m_{j})})}^{2}\\
\nonumber&\leq\underset{j}\liminf\sum_{k=1}^{\infty}\abs{{\eta^{t}_{x}}(1\otimes w_{k}^{(m_{j})})}^{2}\text{ (by Fatou's Lemma)}\\
\nonumber&\leq\sum_{k=1}^{\infty}\abs{{\eta^{t}_{x}}(1\otimes
w_{k})}^{2}<\infty\text{ (as }t\in F^{c}).
\end{align}
\indent Therefore for each $t\in F^{c}$,
\begin{align}\label{infinite}
\underset{\underline{f}\in \mathcal{O}}\sup\text{ }
\sum_{k=1}^{\infty}\abs{{\eta^{t}_{x}}(1\otimes f_{k})}^{2}\leq
\sum_{k=1}^{\infty}\abs{{\eta^{t}_{x}}(1\otimes w_{k})}^{2}<\infty.
\end{align}
\indent Fix $t\in F^{c}$. If $\tilde{\eta}_{x}^{t}$ contains a
nonzero part which is singular with respect to $\lambda$, then the
supremum on the left hand side of Eq. \eqref{infinite} is infinite.
Indeed, for simplicity assume $\tilde{\eta}_{x}^{t}\perp \lambda$.
Choose a compact set $K\subset [0,1]$ of almost full
$\tilde{\eta}_{x}^{t}$ measure such that $\lambda(K)=0$. Fix a large
positive number $N$. By regularity of $\lambda$, there is an open set
$U$ containing $K$ such that $\lambda(U)<\frac{1}{N^{8}}$. Using
compactness of $K$, we can find a finite number of open intervals
$(a_{i},b_{i})$ and small positive numbers $\delta_{i}$ for
$i=1,2,\cdots, m$, such that the open intervals
$\{(a_{i}-\delta_{i}, b_{i}+\delta_{i})\}_{i=1}^{m}$ are disjoint
and $K\subset \cup_{i=1}^{m}(a_{i},b_{i})\
\subset\cup_{i=1}^{m}(a_{i}-\delta_{i}, b_{i}+\delta_{i}) \subset
U$. Define
\begin{equation}
\nonumber f_{i}(s)=\begin{cases}
             N &\text{ if } a_{i}\leq s\leq b_{i},\\
             \frac{N}{\delta_{i}}(s-a_{i})+N &\text{ if }a_{i}-\delta_{i}\leq s\leq a_{i},\\
             -\frac{N}{\delta_{i}}(s-b_{i})+N &\text{ if }b_{i}\leq s\leq b_{i}+\delta_{i},\\
             0 &\text{ otherwise.}
         \end{cases}
\end{equation}
\indent Then $f=\sum_{i=1}^{m}f_{i}$ is continuous and
$\norm{f}_{2,\lambda}=O(\frac{1}{N^{3}})$. Now consider
$g=\frac{f}{\norm{f}_{2,\lambda}}$. Inductively construct an
orthonormal sequence in $C[0,1]$ with the first function as $g$,
orthogonal with respect to the $\lambda$ measure. It is now clear that
in this way the supremum in Eq. \eqref{infinite} can be made to exceed any large number.\\
\indent Consequently, it follows that for all $t\in F^{c}$,
\begin{align}\label{radon_deri_sqint}
\tilde{\eta}_{x}^{t}\ll \lambda.
\end{align}
Finally from Lemma 3.6 of \cite{MR11932708}, it follows that
$\eta_{x}\ll \lambda\otimes \lambda$.
\end{proof}

\begin{rem}\label{derivative_sqint}
Note that the proof of Prop. \ref{product_measure_at_each_step}
actually shows that $\tilde{\eta}_{x}^{t}\ll \lambda$ with
$\frac{d\tilde{\eta}_{x}^{t}}{d\lambda}\in L^{2}([0,1],\lambda)$ for
$\lambda$ almost all $t$.
\end{rem}

The set of finite signed measures on the measurable space
$(X,\sigma_{X})$ is a Banach space equipped with the total variation
norm $\norm{\cdot}_{t.v}$, also called the $L_{1}$-norm, which is
defined by $\norm{\mu}_{t.v}=\abs{\mu}(X)$, where $\abs{\mu}$
denotes the variation measure of $\mu$. It is well known that for
probability measures $\mu$ and $\nu$,
\begin{equation}\label{close_measure_formulae1}
\norm{\mu-\nu}_{t.v}=2\text{ }\underset{B\in \sigma_{X}}\sup
\abs{\mu(B)-\nu(B)}=\int_{X}\abs{f-g}d\gamma
\end{equation}
where $f,g$ are density functions of $\mu,\nu$ respectively with
respect to any $\sigma$-finite measure $\gamma$ dominating both
$\mu,\nu$ $($see for instance Eq. $(1.1)$ of \cite{MR810175}$)$. We are now in a position to prove Thm. \ref{sum_cond}.

\begin{proof}[Proof of Thm. \ref{sum_cond}.]
Fix a set $S\subset \mathcal{M}$ such that $\mathbb{E}_{A}(x)=0$ for
all $x\in S$, $\overline{\text{ span
}S}^{\norm{\cdot}_{2}}=L^{2}(A)^{\perp}$ and
\begin{align}
\nonumber \sum_{k=1}^{\infty}\norm{\mathbb{E}_{A}(x
u_{k}x^{*})}_{2}^{2}<\infty
\end{align}
for all $x\in S$, where $\{u_{k}\}_{k=1}^{\infty}\subset C[0,1]$ is
an orthonormal basis of $L^{2}(A)$. There is a vector $\zeta\in
L^{2}(A)^{\perp}$ such that $\norm{\zeta}_{2}=1$ and
$\eta_{\zeta}=\eta$. Choose a sequence $x_{n}\in
\text{ span }S$ such that $\norm{x_{n}}_{2}= 1$ and
$x_{n}\rightarrow \zeta$ in $\norm{\cdot}_{2}$ as $n\rightarrow
\infty$. Then $($see Lemma 3.10 \cite{MR11932708}$)$, we have
$\eta_{x_{n}}\rightarrow \eta_{\zeta} =\eta\text{ in }
\norm{\cdot}_{t.v}$. Write $x_{n}=\sum_{i=1}^{k_{n}}c_{i,n}y_{i,n}$
with $y_{i,n}\in S$, $c_{i,n}\in \C$ for all $1\leq i\leq k_{n}$ and
$n\in \mathbb{N}$. As $y_{i,n}\in S$, so for all $n\in \mathbb{N}$
and $1\leq i\leq k_{n}$,
\begin{align}
\nonumber
\sum_{k=1}^{\infty}\norm{\mathbb{E}_{A}(y_{i,n}u_{k}y_{i,n}^{*})}_{2}^{2}<\infty.
\end{align}
\indent From Prop. \ref{product_measure_at_each_step} we have
$\eta_{y_{i,n}}\ll \lambda\otimes \lambda$. But
$$\eta_{x_{n}}=\sum_{i=1}^{k_{n}}\abs{c_{i,n}}^{2}\eta_{y_{i,n}}+\sum_{i\neq
j=1}^{k_{n}}c_{i,n}\overline{c}_{j,n}\eta_{y_{i,n},y_{j,n}}.$$
\indent For $1\leq i\neq j\leq k_{n}$, the measures
$\eta_{y_{i,n},y_{j,n}}$ are possibly complex measures, but from Eq.
\eqref{abscont11}, $\abs{\eta_{y_{i,n},y_{j,n}}}\leq
\eta_{y_{i,n}}+\eta_{y_{j,n}}\ll \lambda\otimes \lambda$. Therefore
$\eta_{x_{n}}\ll \lambda\otimes \lambda$. Since $\eta_{x_{n}}$ is a
probability measure so from Eq. \eqref{close_measure_formulae1},
\begin{align}
\nonumber
\frac{1}{2}\norm{\eta_{x_{n}}-\eta_{x_{m}}}_{t.v}=\int_{[0,1]\times
[0,1]}\abs{f_{n}(t,s)-f_{m}(t,s)}d(\lambda\otimes
\lambda)(t,s)\rightarrow 0
\end{align}
as $n,m\rightarrow \infty$, where
$f_{n}=\frac{d\eta_{x_{n}}}{d(\lambda\otimes \lambda)}$. Thus there
is a function $f\in L^{1}([0,1]\times [0,1],\lambda\otimes \lambda)$
such that
\begin{align}
\nonumber \int_{[0,1]\times
[0,1]}\abs{f_{n}(t,s)-f(t,s)}d(\lambda\otimes
\lambda)(t,s)\rightarrow 0
\end{align}
as $n\rightarrow \infty$. As $\eta_{x_{n}}$ is a probability measure
for each $n$, so $\norm{f_{n}}_{L^{1}(\lambda\otimes \lambda)}=1$
for all $n$. Therefore $\norm{f}_{L^{1}(\lambda\otimes \lambda)}=1$
and $\eta_{x_{n}}\rightarrow fd(\lambda\otimes \lambda)$ in
$\norm{\cdot}_{t.v}$. By uniqueness of limits
$\eta=fd(\lambda\otimes \lambda)$.\\
\indent We will now use condition $(iii)$ of Thm. \ref{sum_cond} to
show that $\lambda\otimes \lambda \ll\eta$. Let $v\in A$ be the Haar
unitary corresponding to the function $t\mapsto e^{2\pi it}$. As noted after the statement of Thm. \ref{sum_cond},
we can assume $u=v$ in statement $(iii)$ of Thm. \ref{sum_cond}.
There is a nonzero vector $\xi_{0}\in L^{2}(\mathcal{M})\ominus L^{2}(A)$ such that $\mathbb{E}_{A}(\xi_{0}v^{n}\xi_{0}^{*})=0$ for all $n\neq 0$.
By $3^{\circ}$ of Lemma
\ref{identify_disintegrated_measure} we have
\begin{align}
\nonumber \eta_{\xi_{0}}^{t}(1\otimes v^{n})=0 \text{ for all }n\neq 0
\text{ and for }\lambda \text{ almost all }t.
\end{align}
Thus, the Fourier coefficients of the measure $\tilde{\eta}_{\xi_{0}}^{t}$ are $\tilde{\eta}_{\xi_{0}}^{t}(n)=0$, $n\neq 0$ for $\lambda$ almost all $t$. Thus $\tilde{\eta}_{\xi_{0}}^{t}$ is equal to a multiple of $\lambda$, for
$\lambda$ almost all $t$. 
The scalar above is the total mass of the measure and in this case is $\mathbb{E}_{A}(\xi_{0}\xi_{0}^{*})(t)$ for $\lambda$ almost all $t$ $($see Lemma \ref{identify_disintegrated_measure}$)$. A straight forward calculation shows that
\begin{align}
\nonumber \eta_{\xi_{0}}(a\otimes b)=\int_{[0,1]\times [0,1]}a(t)b(s)\mathbb{E}_{A}(\xi_{0}\xi_{0}^{*})(t)d(\lambda\otimes\lambda)(t,s), \text{ }a,b\in C[0,1].
\end{align}
Thus $\frac{d\eta_{\xi_{0}}}{d(\lambda\otimes \lambda)}=\mathbb{E}_{A}(\xi_{0}\xi_{0}^{*})\otimes 1$. Using Defn. \ref{definition_of_disintegration} it is obvious that $\lambda\otimes \lambda\ll\eta_{\xi_{0}}$ as well. Thus $[\lambda\otimes \lambda]=[\eta_{\xi_{0}}]$.\\
\indent Note that $\overline{A\xi_{0}A}^{\norm{\cdot}_{2}}\subseteq L^{2}(\mathcal{M})\ominus L^{2}(A)$ and $\overline{A\xi_{0}A}^{\norm{\cdot}_{2}}\cong\int_{[0,1]\times[0,1]}^{\oplus}\C_{t,s}d\eta_{\xi_{0}}$, where $\C_{t,s}=\C$ and associated statement about diagonalizabilty of $\mathcal{A}$ holds. There are two cases to consider. If  $\overline{A\xi_{0}A}^{\norm{\cdot}_{2}}= L^{2}(\mathcal{M})\ominus L^{2}(A)$, then $\eta_{\xi_{0}}$ is indeed the \emph{left-right measure} of $A$. In this case there is nothing to prove. If there is a nonzero vector $\xi_{1}\in L^{2}(\mathcal{M})\ominus L^{2}(A)\ominus \overline{A\xi_{0}A}^{\norm{\cdot}_{2}}$, then by Lemma 5.7 \cite{MR2261688}
\begin{align}
\nonumber &\text{either }\overline{A\xi_{0}A}^{\norm{\cdot}_{2}}\oplus \overline{A\xi_{1}A}^{\norm{\cdot}_{2}}\cong\int_{[0,1]\times[0,1]}^{\oplus}\h_{t,s}d(\eta_{\xi_{0}}+\nu), \text{ }0\neq\nu\perp\eta_{\xi_{0}},\\
\nonumber &\text{or }\overline{A\xi_{0}A}^{\norm{\cdot}_{2}}\oplus \overline{A\xi_{1}A}^{\norm{\cdot}_{2}}\cong\int_{[0,1]\times[0,1]}^{\oplus}\h_{t,s}d\eta_{\xi_{0}},
\end{align}
for a $\eta_{\xi_{0}}+\nu$ $($or $\eta_{\xi_{0}})$  measurable field of Hilbert spaces $\h_{t,s}$, along with associated statement about the diagonalizable algebra. In the first case, use direct integrals to conclude that there is a nonzero vector $\tilde{\xi}_{1}\in L^{2}(\mathcal{M})$ $($which is off course orthogonal to $L^{2}(A)\oplus \overline{A\xi_{0}A}^{\norm{\cdot}_{2}})$ such that $\eta_{\tilde{\xi}_{1}}=\nu$. But by an argument analogous to the first part of the proof it follows that $\nu\ll\lambda\otimes\lambda$. This forces that $\overline{A\xi_{0}A}^{\norm{\cdot}_{2}}\oplus \overline{A\xi_{1}A}^{\norm{\cdot}_{2}}\cong\int_{[0,1]\times[0,1]}^{\oplus}\h_{t,s}d\eta_{\xi_{0}}$. Since we have to repeat this argument at most countably many times to exhaust $L^{2}(\mathcal{M})\ominus L^{2}(A)$ $($and in the process the multiplicity function will only change from Lemma 5.7 \cite{MR2261688}$)$, the proof is complete.
\end{proof}

\begin{rem}
The sum in $(ii)$ of Thm. \ref{sum_cond} is the square of the Hilbert Schmidt norm of the operator $\mathbb{E}_{A}(x\cdot x^{*})$. This is precisely the reason the sum is independent of the choice of the orthonormal basis. Note that if
$\sum_{n\in \mathbb{Z}}\norm{\mathbb{E}_{A}(xv^{n}x^{*})}_{2}^{2}<\infty$ and $\sum_{n\in \mathbb{Z}}\norm{\mathbb{E}_{A}(yv^{n}y^{*})}_{2}^{2}<\infty$, where $x,y\in \mathcal{M}$, $\mathbb{E}_{A}(x)=0, \mathbb{E}_{A}(y)=0$ and $v$ is the standard Haar unitary generator of $A$, then $\sum_{n\in \mathbb{Z}}\norm{\mathbb{E}_{A}(xv^{n}y^{*})}_{2}^{2}<\infty$. To see this, one has to use the facts  $\frac{d{\eta_{x,y}}}{d(\lambda\otimes\lambda)},\frac{d{\eta_{x}}}{d(\lambda\otimes\lambda)},\frac{d{\eta_{y}}}{d(\lambda\otimes\lambda)}\in L^{2}(\lambda\otimes\lambda)$. Thus conditions $(i),(ii)$ in Thm. \ref{sum_cond} can be strengthened as: There exists a set $D\subset \mathcal{M}$ such that
$\mathbb{E}_{A}(x)=0$ for all $x\in D$, the set $D$ is dense in
$L^{2}(A)^{\perp}$ and
\begin{align}
\nonumber \sum_{k=1}^{\infty}\norm{\mathbb{E}_{A}(x_{1}
v_{k}x_{2}^{*})}_{2}^{2}<\infty \text{ for all }x_{1},x_{2}\in D,
\end{align}
for some orthonormal basis $\{v_{k}\}\subset C[0,1]$ of $L^{2}(A)$. One choice of $D$ is $\text{span }S$.
\end{rem}

When the \emph{left-right-measure} of a masa is the class of
product measure, the masa satisfies conditions very close to the ones
described in Thm. \ref{sum_cond}. This is the content of the next
theorem.\\
\indent For $\mathbb{N}_{\infty}\ni n\in Puk(A)$, let
$E_{n}\subseteq [0,1]\times [0,1]\setminus \Delta([0,1])$ denote the
set where the multiplicity function in $m.m(A)$ takes the value $n$.
It is well known that $E_{n}$ is $\eta$-measurable. Then
\begin{align}\label{chop_down_measure}
 L^{2}(\mathcal{M})\ominus L^{2}(A)&\cong\underset{n\in Puk(A)}\oplus L^{2}(E_{n},\eta_{\mid E_{n}})\otimes \mathbb{C}^{n}
\cong\underset{n\in Puk(A)}\oplus\int_{[0,1]\times
[0,1]}^{\oplus}\mathbb{C}^{n}_{t,s}d\eta_{\mid E_{n}}(t,s),
\end{align}
where $\mathbb{C}^{n}_{t,s}=\mathbb{C}^{n}$ for $(t,s)\in E_{n}$
when $n<\infty$, and
$\mathbb{C}^{\infty}=\mathbb{C}^{\infty}_{t,s}=l_{2}(\mathbb{N})$.
Under this decomposition one has
\begin{align}
\nonumber \A^{\prime}(1-e_{A})\cong \underset{n\in Puk(A)}\oplus
L^{\infty}(E_{n},\eta_{\mid E_{n}})\overline{\otimes} \mathcal{M}_{n}(\C),
\end{align}
where $\mathcal{M}_{\infty}(\C)$ is to be interpreted as
$\mathbf{B}(l_{2}(\mathbb{N}))$. Consequently, it follows that for
$\mathbb{N}_{\infty}\ni n\in Puk(A)$ the projections
$\chi_{E_{n}}\otimes 1_{n}$ lie in $\textbf{Z}(\A^{\prime})=\A$,
where $1_{n}$ denotes the identity of $\mathcal{M}_{n}(\C)$ if
$n<\infty$ and $1_{\infty}=1_{\textbf{B}(l_{2}(\mathbb{N}))}$. For
$n\in Puk(A)$ choose vectors $\zeta^{(n)}_{i}$, $1\leq i\leq n$
$(1\leq i<\infty$ if $n=\infty)$ so that the projections
$P_{i}^{(n)}:L^{2}(\mathcal{M})\mapsto
\overline{A\zeta^{(n)}_{i}A}^{\norm{\cdot}_{2}}$ are mutually
orthogonal, equivalent in $\A^{\prime}$, and
$\sum_{i=1}^{n}P_{i}^{(n)}=\chi_{E_{n}}\otimes 1_{n}$.

\begin{thm}\label{measure_to_sum_all_puk}
Let $A\subset \mathcal{M}$ be a masa. Let the left-right-measure of
$A$ be the class of product measure. Then there is a set $S\subset
L^{2}(\mathcal{M})\ominus L^{2}(A)$ such that $\text{ span }S$ is
dense in $L^{2}(A)^{\perp}$,
\begin{align}
\nonumber \sum_{n=1}^{\infty}\norm{\mathbb{E}_{A}(\zeta
w_{n}\zeta^{*})}_{2}^{2}<\infty \text{ for all }\zeta\in S
\end{align}
for some orthonormal basis $\{w_{n}\}_{n=1}^{\infty}\subset A$ of
$L^{2}(A)$, and there is a nonzero $\xi\in L^{2}(\mathcal{M})\ominus
L^{2}(A)$ such that $\mathbb{E}_{A}(\xi v^{n}\xi^{*})=0$ for all
$n\neq 0$, where $v$ is a Haar unitary generator of $A$.
\end{thm}

\begin{proof}
We will first consider the case $Puk(A)=\{1\}$. In this case,
\begin{align}
\nonumber L^{2}(\mathcal{M})\ominus L^{2}(A)\cong L^{2}([0,1]\times
[0,1]\setminus \Delta([0,1]),\lambda\otimes \lambda),
\end{align}
the left and the right actions of $A$ being given by
\begin{align}
\nonumber (af)(t,s)=a(t)f(t,s), (fb)(t,s)=b(s)f(t,s)
\end{align}
where $f\in L^{2}(A)^{\perp}$ and $a,b\in A$.\\
\indent Let $0\neq \zeta\in L^{2}(A)^{\perp}$ be a continuous
function. Then for $a,b\in C[0,1]$,
\begin{align}
\nonumber \langle a\zeta b,\zeta\rangle_{L^{2}(\mathcal{M})}=\langle a\zeta b,\zeta\rangle_{L^{2}(\lambda\otimes \lambda)}&=\int_{[0,1]\times[0,1]}a(t)b(s)\zeta(t,s)\overline{\zeta(t,s)}d\lambda(t)d\lambda(s)\\
\nonumber
&=\int_{[0,1]\times[0,1]}a(t)b(s)\abs{\zeta(t,s)}^{2}d\lambda(t)d\lambda(s).
\end{align}
\indent Therefore $\frac{d\eta_{\zeta}}{d(\lambda\otimes
\lambda)}=\abs{\zeta}^{2}$ which is bounded, in particular in
$L^{2}(\lambda\otimes \lambda)$. We claim that $\mathbb{E}_{A}(\zeta
b\zeta^{*})\in L^{2}(A)$ for any $b\in C[0,1]$. Fix $a\in C[0,1]$.
Then as $\tau$ extends to $L^{1}$,
\begin{align}\label{norm_cal}
\int_{0}^{1}a(t)\mathbb{E}_{A}(\zeta b\zeta^{*})(t)d\lambda(t)=\tau(a\mathbb{E}_{A}(\zeta b\zeta^{*}))=\tau(a\zeta b\zeta^{*})&=\int_{[0,1]\times[0,1]}a(t)b(s)d\eta_{\zeta}(t,s)\\
\nonumber&=\int_{[0,1]\times[0,1]}a(t)b(s)\abs{\zeta}^{2}(t,s)d\lambda(t)d\lambda(s)\\
\nonumber&=\int_{0}^{1}a(t)\lambda(\abs{\zeta}^{2}(t,\cdot)b)d\lambda(t).
\end{align}
\indent Now consider the function $[0,1]\ni t\overset{g}\mapsto
\lambda(\abs{\zeta}^{2}(t,\cdot)b)$. It is clearly
$\lambda$-measurable and
\begin{align}
\nonumber\int_{0}^{1}\abs{\lambda(\abs{\zeta}^{2}(t,\cdot)b)}^{2}d\lambda(t)=\int_{0}^{1}\abs{\int_{0}^{1}\abs{\zeta}^{2}(t,s)b(s)d\lambda(s)}^{2}d\lambda(t)&\leq\norm{b}^{2}\int_{0}^{1}\left(\int_{0}^{1}\abs{\zeta}^{2}(t,s)d\lambda(s)\right)^{2}d\lambda(t)\\
\nonumber&\leq\norm{b}^{2}\int_{0}^{1}\int_{0}^{1}\abs{\zeta}^{4}(t,s)d\lambda(t)d\lambda(s)<\infty.
\end{align}
Therefore from Eq. \eqref{norm_cal} we get,
\begin{align}
\nonumber\underset{a\in C[0,1], \norm{a}_{2}\leq 1}\sup\abs{\int_{0}^{1}a(t)\mathbb{E}_{A}(\zeta b\zeta^{*})(t)d\lambda(t)}&=\underset{a\in C[0,1], \norm{a}_{2}\leq 1}\sup\abs{\int_{0}^{1}a(t)\lambda(\abs{\zeta}^{2}(t,\cdot)b)d\lambda(t)}\\
\nonumber
&=\left(\int_{0}^{1}\abs{\lambda(\abs{\zeta}^{2}(t,\cdot)b)}^{2}d\lambda(t)\right)^{\frac{1}{2}}<\infty.
\end{align}
\indent Consequently, it follows that $\mathbb{E}_{A}(\zeta
b\zeta^{*})\in L^{2}(A)$ and
\begin{align}
\nonumber \norm{\mathbb{E}_{A}(\zeta
b\zeta^{*})}_{2}^{2}=\int_{0}^{1}\abs{\lambda(\abs{\zeta}^{2}(t,\cdot)b)}^{2}d\lambda(t).
\end{align}
\indent Let $v\in A$ be the Haar unitary corresponding to the
function $t\mapsto e^{2\pi it}$. Then $\{v^{n}\}_{n\in \mathbb{Z}}$
is an orthonormal basis of $L^{2}(A)$ and by Parseval's theorem,
\begin{align}
\nonumber \sum_{n\in\mathbb{Z}}\norm{\mathbb{E}_{A}(\zeta v^{n}\zeta^{*})}_{2}^{2}=\sum_{n\in\mathbb{Z}}\int_{0}^{1}\abs{\lambda(\abs{\zeta}^{2}(t,\cdot)v^{n})}^{2}d\lambda(t)&=\int_{0}^{1}\sum_{n\in\mathbb{Z}}\abs{\lambda(\abs{\zeta}^{2}(t,\cdot)v^{n})}^{2}d\lambda(t)\\
\nonumber
&=\int_{0}^{1}\int_{0}^{1}\abs{\zeta}^{4}(t,s)d\lambda(s)d\lambda(t)<\infty.
\end{align}
Thus $\{\zeta\in
L^{2}(A)^{\perp}:\sum_{n\in\mathbb{Z}}\norm{\mathbb{E}_{A}(\zeta
v^{n}\zeta^{*})}_{2}^{2}<\infty \}$ is dense in $L^{2}(A)^{\perp}$.\\
\indent In the general case, write
\begin{align}
\nonumber L^{2}(\mathcal{M})\ominus L^{2}(A)=\underset{n\in
Puk(A)}\oplus\left(\overset{n}{\underset{i=1}\oplus}\overline{A\zeta^{(n)}_{i}A}^{\norm{\cdot}_{2}}\right),
\end{align}
where $\zeta_{i}^{(n)}$ are vectors defined prior to the proof. For
each $n\in Puk(A)$ and $1\leq i\leq n$ $($or $1\leq i<n$ as the case
may be$)$, we consider the left and right actions of $A$ on
$\overline{A\zeta^{(n)}_{i}A}^{\norm{\cdot}_{2}}$ to reduce the
problem to a case similar to having one bicyclic
vector. In this case, one works with bounded measurable functions.\\
\indent Finally, let $\zeta\in L^{2}(\mathcal{M})$ correspond to the
function $\chi_{\{(t,s):t\neq s\}}$. Then
$\eta_{\zeta}=\lambda\otimes \lambda$. By arguments exactly similar
to the first part of the proof, conclude that $\mathbb{E}_{A}(\zeta
a\zeta ^{*})\in L^{2}(A)$ for all $a\in A$. But by $3^{\circ}$ of
Lemma \ref{identify_disintegrated_measure} we get,
\begin{align}
\nonumber \norm{\mathbb{E}_{A}(\zeta v^{n}\zeta
^{*})}_{2}^{2}=\int_{0}^{1}\abs{\eta_{\zeta}^{t}(1\otimes
v^{n})}^{2}d\lambda(t)=0 \text{ for all }n\neq 0.
\end{align}
\end{proof}

The proof of Thm. \ref{sum_cond} and Thm. \ref{measure_to_sum_all_puk} yield the following corollary.

\begin{cor}\label{leb_confirm}
For a masa $A\subset \mathcal{M}$, the
left-right-measure of $A$ contains the product class as a summand, if and only if,
there is a nonzero $\xi\in L^{2}(\mathcal{M})\ominus L^{2}(A)$
such that $\mathbb{E}_{A}(\xi v^{n}\xi^{*})=0$ for all $n\neq 0$,
where $v$ is a Haar unitary generator of $A$.
\end{cor}
\begin{proof}
$\Rightarrow$ By Lemma 5.7 \cite{MR2261688}, the \emph{left-right-measure} of $A$ is of the form $[\lambda\otimes\lambda+\nu]$ where either $\nu=0$ or $\nu\perp\lambda\otimes\lambda$. In any case, there is a nonzero vector $\xi\in L^{2}(A)^{\perp}$ such that $\eta_{\xi}=\lambda\otimes\lambda$. Now use the argument of last part of Thm. \ref{measure_to_sum_all_puk}. The reverse direction follows from the proof of Thm. \ref{sum_cond}.
\end{proof}

\begin{rem}
The proof of the previous theorem shows that if $A\subset \mathcal{M}$ is a masa satisfying the conditions of Thm. \ref{sum_cond} or the \emph{left-right-measure} of $A$ is the product class, then there is a measurable partition $\{E_{n}\}_{n\in Puk(A)}$ of $\Delta([0,1])^{c}$ such that
\begin{align}
\nonumber {}_{A}L^{2}(\mathcal{M})\ominus L^{2}(A)_{A}\cong \underset{n\in Puk(A)}\oplus \oplus_{i=1}^{n} {}_{A}L^{2}(E_{n},\lambda\otimes\lambda)_{A}
\end{align}
with the natural actions on the right hand side.
\end{rem}

Note that the \emph{measure-multiplicity invariant} can be defined for any diffuse abelian subalgebra of $\mathcal{M}$ in exactly the similar way defined in Defn. \ref{mminv}. If the diffuse abelian algebra is not a masa, then the diagonal will correspond to the $L^{2}$ completion of the relative commutant of the abelian algebra, so the multiplicity function along the diagonal will not be constantly $1$. All other properties of the invariant will remain the same. We will use this observation in the following result.

\begin{thm}
Let $A\subset \mathcal{M}$ be a masa such that the left-right-measure of $A$ is the class of product measure. Then for any diffuse algebra $B\subset A$, the left right-measure of $B$ restricted to the off-diagonal is the class of product measure and $N(B)^{\prime\prime}=B^{\prime}\cap \mathcal{M}= A$.
\end{thm}

\begin{proof}
Since the \emph{left-right-measure} of $A$ is $[\lambda\otimes\lambda]$, so by Thm. \ref{measure_to_sum_all_puk}, there is a set $S$ orthogonal to $L^{2}(A)$, such that $\text{span }S$ is dense in $L^{2}(A)^{\perp}$, $\sum_{n\in \mathbb{Z}}\norm{\mathbb{E}_{A}(\zeta v^{n}\zeta^{*})}_{2}^{2}<\infty$ for all $\zeta\in S$, where $v$ is the standard Haar unitary generator of $A$. Moreover, the proof of Thm. \ref{measure_to_sum_all_puk} shows that we can assume $\frac{d\eta_{\zeta}}{d(\lambda\otimes\lambda)}$ is bounded $\lambda\otimes\lambda$ almost everywhere.\\
\indent Arguments similar to the proof of Thm. \ref{measure_to_sum_all_puk} show that $\mathbb{E}_{A}(\zeta\cdot\zeta^{*})$ defines a Hilbert Schmidt operator on $L^{2}(A)$. Fix a diffuse subalgebra $B\subset A$. Let $w\in B$ be a Haar unitary generator of $B$. Since $\mathbb{E}_{A}(\zeta\cdot\zeta^{*})$ is Hilbert Schmidt, so $\sum_{n\in \mathbb{Z}}\norm{\mathbb{E}_{A}(\zeta w^{n}\zeta^{*})}_{2}^{2}<\infty$ and since $\norm{\mathbb{E}_{B}(\cdot)}_{2}\leq \norm{\mathbb{E}_{A}(\cdot)}_{2}$ so
\begin{align}
\nonumber \sum_{n\in \mathbb{Z}}\norm{\mathbb{E}_{B}(\zeta w^{n}\zeta^{*})}_{2}^{2}<\infty, \text{for all }\zeta\in S.
\end{align}
\indent Assuming $B=L^{\infty}([0,1],\lambda)$ where $\lambda$ is Lebesgue measure and using arguments required to prove Prop. \ref{product_measure_at_each_step}, one finds $\eta_{\zeta,B}\ll\lambda\otimes\lambda$ for all $\zeta\in S$. The extra suffix refers to the fact that we are considering measures with respect to $B$. (It should be noted that the proof of Prop. \ref{product_measure_at_each_step} nowhere uses the fact that $A$ is a masa.) Thus $\mathbb{E}_{B^{\prime}\cap \mathcal{M}}(\zeta)=0$ for all $\zeta\in S$. Indeed, write $\zeta=\zeta_{1}+\zeta_{2}$ with
$\mathbb{E}_{B^{\prime}\cap \mathcal{M}}(\zeta)=\zeta_{1}$ and $\mathbb{E}_{B^{\prime}\cap \mathcal{M}}(\zeta_{2})=0$. For $a,b\in B$ one has
\begin{align}
\nonumber\langle a\zeta_{1}b,\zeta_{2}\rangle=\tau(a\zeta_{1}b\zeta_{2}^{*})=\tau(\mathbb{E}_{B^{\prime}\cap \mathcal{M}}(a\zeta_{1}b)\zeta_{2}^{*})=0.
\end{align}
Thus $\eta_{\zeta,B}=\eta_{\zeta_{1,B}}+\eta_{\zeta_{2,B}}$. But $\eta_{\zeta_{1,B}}\ll\tilde{\Delta}_{*}\lambda$ (for $\tilde{\Delta}$ see \S 1) with the Radon-Nikodym derivative given by $\mathbb{E}_{B}(\zeta_{1}\zeta_{1}^{*})$. Consequently, $\zeta_{1}=0$. Thus $S\subset L^{2}(B^{\prime}\cap \mathcal{M})^{\perp}$ and hence $L^{2}(A)^{\perp}\subseteq L^{2}(B^{\prime}\cap \mathcal{M})^{\perp}$. It follows that $B^{\prime}\cap \mathcal{M}=A$.\\
\indent By arguments similar to the proof of Thm. \ref{sum_cond}, it follows that any member in \emph{left-right-measure} of $B$ restricted to the off-diagonal is dominated by $\lambda\otimes\lambda$.\\
\indent There is a vector $0\neq\xi\in L^{2}(A)^{\perp}$ such that $\mathbb{E}_{A}(\xi v^{n}\xi^{*})=0$ for all $n\neq 0$. It follows that $\mathbb{E}_{A}(\xi a\xi^{*})=0$ for all $a\in A$ with $\tau(a)=0$. Consequently, $\mathbb{E}_{A}(\xi w^{n}\xi^{*})=0$ and hence $\mathbb{E}_{B}(\xi w^{n}\xi^{*})=0$ for all $n\neq 0$. By arguments made in the last part of the proof of Thm. \ref{sum_cond}, it follows that the \emph{left-right-measure} of $B$ restricted to the off diagonal is the class of product measure.\\
\indent Finally, if $0\neq\zeta_{0}\in L^{2}(N(B)^{\prime\prime})$, then $\overline{B\zeta_{0}B}^{\norm{\cdot}_{2}}\in C_{d}(B)$ (Prop. 3.11 \cite{MR11932708}). Thus by using Lemma 5.7 \cite{MR2261688}, it follows $\eta_{\zeta_{0,B}}$ must be supported on the diagonal; equivalently $\mathbb{E}_{B^{\prime}\cap \mathcal{M}}(\zeta_{0})=\zeta_{0}$. Thus $\zeta_{0}\in L^{2}(A)$. This completes the proof.
\end{proof}

\section{Tauer Masas in the Hyperfinite $\rm{II}_{1}$ Factor}

In this section, we will calculate the \emph{left-right-measures} of
certain Tauer masas in the hyperfinite $\rm{II}_{1}$ factor
$\mathcal{R}$. The examples of Tauer masas in which we are
interested are directly taken from \cite{MR2302742}.

\begin{defn}(White)
A masa $A$ in $\mathcal{R}$ is said to be a \emph{Tauer masa}, if
there exists a sequence of finite type $\rm{I}$ subfactors
$\{\mathcal{N}_{n}\}_{n=1}^{\infty}$ such that,\\
$(i)$ ${\mathcal{N}_{n}} \subset {\mathcal{N}_{n+1}}$ for all $n$,\\
$(ii)$$(\cup _{n=1}^{\infty}{\mathcal{N}_{n}})^{\prime\prime} = \mathcal{R}$,\\
$(iii)$ $A_{n}=A\cap {\mathcal{N}_{n}}$ is a masa in ${\mathcal{N}_{n}}$
for every $n$.
\end{defn}
This allows one to write the structure of every Tauer masa $A$ in
$\mathcal{R}$ with respect to the chain
$\{\mathcal{N}_{n}\}_{n=1}^{\infty}$ as follows. Switching to the
notation of tensor products, the above definition means that we can
find finite type $\rm{I}$ subfactors
$\{{\mathcal{M}}_{n}\}_{n=1}^{\infty}$ such that, ${\mathcal{N}_{n}=
\overset n{\underset {r=1}\otimes} {\mathcal{M}}_{r}}$ for every
$n$. For $m > n$, the $m$-th finite dimensional approximation of $A$
can be written in terms of the $n$-th one as,
\begin{equation}\label{Tauerst}
A_{m}=\bigoplus_{e\in \mathcal{P}(A_{n})}e \otimes A_{m,n}^{(e)},
\end{equation}
where the direct sum is over the set of minimal projections
$\mathcal{P}(A_{n})$ in $A_{n}$ and $A_{m,n}^{(e)}$ is a masa in
$\overset m{\underset{r=n+1}\otimes} {\mathcal{M}_{r}}$. Note that
the Cartan masa arising from the infinite tensor product of diagonal
matrices inside the hyperfinite $\rm{II}_{1}$ factor is a Tauer
masa. In Thm. 4.1 \cite{MR2253595}, White had shown that the
\emph{Puk{\'a}nszky invariant} of every Tauer masa is $\{1\}$. In
fact, it follows from his proof that the bicyclic vector for any
Tauer masa can be chosen to be an operator from $\mathcal{R}$ itself.\\
\indent Sinclair and White \cite{MR2302742} has exhibited a
continuous path of singular masas in $\mathcal{R}$, no two of which
can be connected by automorphisms of $\mathcal{R}$. We are
interested in two masas that correspond to the end points of this
path. For all Tauer masas, it is clear that the Cantor set is the
natural space where we have to build the measures. For ease of
calculation, we need to index the minimal projections in the
approximating stages in an appropriate fashion.
It is now time to introduce some notation.\\

\noindent $1^{\circ}\text{ }\mathbf{Notation:}$ If $\mathcal{N}_{n}=
\overset n {\underset{r=1}\otimes}
{\mathcal{M}}_{k_{r}}(\mathbb{C})$, then the minimal projections of
$A_{n}$ will be denoted by ${}^{(n)}f_{\underline{t}(n)}$ , where
$\underline{t}(n)=(t_{1},t_{2},\cdots,t_{n})$ with $1\leq t_{i}\leq
k_{i}$, $1\leq i\leq n$. The convention that we follow is
$${}^{(n)}f_{(t_{1},t_{2},\cdots,t_{n})}={}^{(n-1)}f_{(t_{1},t_{2},\cdots,t_{n-1})}\otimes {}^{(n)}e^{(t_{1},t_{2},\cdots,t_{n-1})}_{t_{n}},$$
where ${}^{ (n)}e^{(t_{1},t_{2},\cdots,t_{n-1})}_{t_{n}}$ are the
minimal projections of the algebra
$A^{(t_{1},t_{2},\cdots,t_{n-1})}_{n,n-1}$, in accordance with Eq.
\eqref{Tauerst}. The matrix units corresponding to this family of
minimal projections will be denoted by
${}^{(n)}f_{\underline{t}(n),\text{ }\underline{s}(n)}$ and we will
understand ${}^{(n)}f_{\underline{t}(n),\text{
}\underline{t}(n)}={}^{(n)}f_{\underline{t}(n)}$. For two tuples
$(t_{1},t_{2},\cdots,t_{n})$ and $(s_{1},s_{2},\cdots,s_{n})$ such
that $t_{i}=s_{i}\text{ for }1\leq i\leq n-1$ and $t_{n}\neq s_{n}$,
we will write ${}^{(n)}f_{\underline{t}(n),\text{
}\underline{s}(n)}={}^{(n)}f_{(\cdot,t_{n}),(\cdot,s_{n})}$.\\

\noindent $2^{\circ}$ \textbf{Notation:} For any two subsets $S,T
\subseteq \mathcal{M}$, we will denote by $S\cdot T$ the set
$span\{ab : a$ $\in$ $S,$$b$ $\in$ $T$$\}$. The normalized trace of
$\mathcal{M}_{n}(\C)$ will be denoted by $tr_{n}$. The unique normal
tracial state of the hyperfinite factor $\mathcal{R}$ will be
denoted by $\tau_{\mathcal{R}}$. This trace $\tau_{\mathcal{R}}$
when restricted to $A$ gives rise to a measure on a Cantor set which
will also be denoted by $\tau_{\mathcal{R}}$.\\
\indent Recall from \cite{MR703810} that
two subalgebras $B,C$ in a finite factor $N$ are called orthogonal with respect
to the unique normal tracial state
$\tau_{N}$, if $\tau_{N}(bc)=\tau_{N}(b)\tau_{N}(c)$ for all $b\in B$, $c\in C$.
The next lemma is very well known but we record it for
convenience.
\begin{lem}\label{lemma:orth}
If A, B are two masas in $\mathcal{M}_{n}(\C)$ orthogonal with
respect to the normalized trace $tr_{n}$, then $A\cdot
B=\mathcal{M}_{n}(\C)$.
\end{lem}

\subsection{Tauer Masa of Product Class}
$ $\\\\
\indent Following Sinclair and White \cite{MR2302742}, we shall
calculate the \emph{measure-multiplicity invariant} of a Tauer
masa $A$, whose description is elaborated below. The
$\Gamma $ invariant of this Tauer masa is $0$ (A is totally
non-$\Gamma$ \cite{MR2302742}). We will show that its \emph{left-right-measure}
belongs to the product class. This example is important, as, it is
an example of a masa in $\mathcal{R}$ with simple multiplicity whose
\emph{left-right-measure} is the class of product measure. Such masas
are rare in $\mathcal{R}$. We do not know whether it arises from a
dynamical system.\\
\indent Let $k_{1}=2$, and for each $r\geq 2$, let $k_{r}$ be a prime
exceeding $k_{1}k_{2}\cdots k_{r-1}$. Set $\mathcal{M}_{r}$ to be
the algebra of $k_{r} \times k_{r}$ matrices. By Thm. 3.2
\cite{MR703810}, there is a family
$\{{}^{(r)}D^{\underline{t}(r-1)}\}_{\underline{t}(r-1)}$ of
pairwise orthogonal masas in $\mathcal{M}_{r}$. Let
$\mathcal{N}_{n}= \overset
n{\underset{r=1}\bigotimes}\mathcal{M}_{r}$. There is a natural
inclusion $x \mapsto x\otimes 1$ of $\mathcal{N}_{n}$ inside
$\mathcal{N}_{n+1}$ and one works in the hyperfinite $\rm{II}_{1}$
factor $\mathcal{R}$, obtained as a direct limit of these
$\mathcal{N}_{n}$ with respect to the normalized trace. With respect
to the chain $\{\mathcal{N}_{n}\}_{n=1}^{\infty}$ of finite type
$\rm{I}$ subfactors of
$\mathcal{R}$, the masa $A$ is constructed as follows.\\
\indent Let $A_{1}=D_{2}(\mathbb{C})\subset \mathcal{M}_{1}$ be the diagonal
masa. Having constructed $A_{n}$, one constructs $A_{n+1}$ as,
\begin{equation}
A_{n+1}= \underset{\underline{t}(n)}\bigoplus
{}^{(n)}f_{\underline{t}(n)} \otimes
{}^{(n+1)}D^{{\underline{t}(n)}}.
\end{equation}
That $(\cup_{n=1}^{\infty}A_{n})^{\prime\prime}$ is a masa in
$\mathcal{R}$, follows from a theorem of Tauer $($see Thm. 2.5 \cite{MR0182892}$)$.
This
Tauer masa is singular from Prop. 2.1 \cite{MR2302742}.\\
\indent We denote by $P_{\underline{t}(n),\underline{s}(n)}^{(n)}$
the orthogonal projection from $L^{2}(\mathcal{R})$ onto the
subspace
${}^{(n)}f_{\underline{t}(n)}L^{2}(\mathcal{R}){}^{(n)}f_{\underline{s}(n)}$,
and let,
\begin{equation}\label{proj1}
P=\sum _{n=1}^{\infty}\underset{t_{1}=s_{1},\cdots, t_{n-1}=s_{n-1},t_{n}\neq s_{n}}{\underset{\underline{t}(n),\text{ }\underline{s}(n):}\sum}P_{(t_{1},\cdots,t_{n}),(s_{1},\cdots,s_{n})}^{(n)}.
\end{equation}
\indent Clearly,
$P_{\underline{t}(n),\underline{s}(n)}^{(n)}={}^{(n)}f_{\underline{t}(n)}J{}^{(n)}f_{\underline{s}(n)}J$
and is in $\mathcal{A}$. At the first sight, it might not be clear
that the sum in Eq. \eqref{proj1} makes sense, but, the
projections involved in the sum are orthogonal and sums to
$1-e_{A}$. Indeed, since $e_{A}$ is the limit in strong operator topology of $e_{A_{n}^{\prime}\cap \mathcal{R}}=\sum_{\underline{t}(n)}P^{(n)}_{(t_{1},t_{2},\cdots, t_{n}),(t_{1},t_{2},\cdots, t_{n})}$ $($\S 5.3 \cite{MR1336825}, Lemma 1.2 \cite{MR641131}$)$, that $P=1-e_{A}$ follows by rearranging terms in Eq. \eqref{proj1}.\\
\indent The following lemma, part of which was recorded by Sinclair
and White \cite{MR2302742}, will be crucial for our calculations.

\begin{lem}\label{lemma:A}
For each $n$ $\in$ $\mathbb {N}$, let $\mathcal{R}=
{\mathcal{N}}_{n}\bigotimes {\mathcal{R}}_{n}$, where
$\mathcal{R}_{n} =( \overset{\infty}{\underset{r=n+1}\bigotimes}
\mathcal{M}_{k_{r}}(\mathbb{C}))^{\prime\prime}$. Then
\begin{equation}\label{decompose}
A=\underset{\underline{t}(n)}\bigoplus {}^{(n)}f_{\underline{t}(n)}
\otimes A_{\infty , n+1}^{\underline{t}(n)}, \text{ where}
\end{equation}
$A_{\infty , n+1}^{\underline{t}(n)}$ are Tauer masas in
${\mathcal{R}}_{n}$ and whenever $\underline{t}(n)\neq
\underline{s}(n)$
we have \\
$(i)$ $A_{\infty , n+1}^{\underline{t}(n)}$ and $A_{\infty ,
n+1}^{\underline{s}(n)}$ are orthogonal in
${\mathcal{R}}_{n}$,\\
$(ii)$ $(A_{\infty , n+1}^{\underline{t}(n)}\cdot A_{\infty ,
n+1}^{\underline{s}(n)})$$^{- \parallel . \parallel
_{2}}=L^{2}(\mathcal{R}_{n}).$ \\
Moreover, for  each  $\underline{t}(n)$ if $\{A_{m ,
n+1}^{\underline{t}(n)}\}_{m=1}^{\infty}$ denote the $m$-th
approximation of $A_{\infty ,
n+1}^{\underline{t}(n)}$ in $\mathcal{R}_{n}$, then\\
\begin{equation}\label{equation:orth}
A_{1 , n+1}^{\underline{t}(n)} ={}^{(n+1)}D^{\underline{t}(n)},\text{
 and, }
\end{equation}
\begin{equation}
A_{m+1 , n+1}^{\underline{t}(n)}=\bigoplus_{ e \in \mathcal{P}(A_{m
, n+1}^{\underline{t}(n)})} e \otimes
^{(m+1)}D_{e,n+1}^{\underline{t}(n)},
\end{equation}
where for each fixed $m$
and $\underline{t}(n)$, the family
$\{^{(m+1)}D_{e,n+1}^{\underline{t}(n)}\}_{e}$ are pairwise
orthogonal masas in $\mathcal{M}_{k_{n+m+1}}(\mathbb{C})$.
\end{lem}
\begin{proof}
It should be understood that in $(ii)$ of the statement, the closure is taken with respect to the faithful normal tracial state of $\mathcal{R}_{n}$. We only have to prove $(ii)$. The rest of the statements are just
rephrasing of Lemma 5.6 of \cite{MR2302742}.\\
\indent Use Lemma \ref{lemma:orth}, $(i)$ and Eq. \eqref{equation:orth} to
conclude that $$\mathcal{M}_{k_{n+1}}\subseteq (A_{\infty ,
n+1}^{\underline{t}(n)}\cdot A_{\infty ,
n+1}^{\underline{s}(n)})^{-\norm{\cdot}_{2}}.$$ Since $A_{\infty ,
n+1}^{\underline{t}(n)}$ and $A_{\infty , n+1}^{\underline{s}(n)}$
are orthogonal, so is $A_{m , n+1}^{\underline{t}(n)}$ and $A_{m ,
n+1}^{\underline{s}(n)}$ for all $m\geq n+1$. Use Lemma
\ref{lemma:orth} to conclude that
$\overset{m}{\underset{r=n+1}\bigotimes}\mathcal{M}_{k_{r}}(\C)\subseteq
(A_{\infty , n+1}^{\underline{t}(n)}\cdot A_{\infty ,
n+1}^{\underline{s}(n)})^{-\norm{\cdot}_{2}}$ for all $m\geq n+1$.
Now use density of algebraic tensor product of matrix algebras
in $L^{2}(\mathcal{R}_{n})$ to finish the proof.
\end{proof}

\indent For each $n$, let $X_{n}=\{
x_{{1}}^{(n)},x_{{2}}^{(n)},\cdots,x_{k_{n}}^{(n)}\}$ denote a set
of $k_{n}$ points. Let $Y^{(n)}=\overset n{\underset{k=1}\prod}
X_{k}$, $X^{(n)}=\overset {\infty}{\underset{k=n+1}\prod}X_{k}$ and
$X=\overset {\infty}{\underset{k=1}\prod}X_{k}$, so that for each
$n$, $X=Y^{(n)} \times X^{(n)}$. Therefore,
$X=\underset{{\infty\longleftarrow n}}{\lim} Y^{(n)}$ and $C(X)$ is
norm separable and \emph{w.o.t} dense in $A$. The identification is
a standard one and we omit the details. Write $B=C(X)$. Therefore,
\begin{equation}
B=\underset{\underline{t}(n)}\bigoplus{}^{(n)}f_{\underline{t}(n)}\otimes
B^{\underline{t}(n)}_{\infty,n+1},
\end{equation}
$B^{\underline{t}(n)}_{\infty,n+1}\cong C(X^{(n+1)})$ and is a
\emph{w.o.t} dense, norm separable $C^{*}$ subalgebra of $A_{\infty
, n+1}^{\underline{t}(n)}$.

\begin{lem}\label{lemma:prodm}
For each $n$ and
$\underline{t}(n) \neq \underline{s}(n)$,\\
$(i)(A{}^{(n)}f_{\underline{t}(n),\underline{s}(n)}A)^{-\norm{.}_{2}}={}^{(n)}f_{\underline{t}(n)}L^{2}(\mathcal{R}){}^{(n)}f_{\underline{s}(n)}$,\\
$(ii)$ for $a$, $b$ $\in$ $B$,
\begin{align}
\nonumber &\langle a{}^{(n)}f_{\underline{t}(n),\text{
}\underline{s}(n)}b,{}^{(n)}f_{\underline{t}(n),\text{
}\underline{s}(n)}\rangle_{\tau_{ \mathcal{R}}}=& k_{1}k_{2}\cdots
k_{n}\int_{X}\int_{X}{}^{(n)}f_{\underline{t}(n)}(t){}^{(n)}f_{\underline{s}(n)}(s)a(t)b(s)d(\tau_{
\mathcal{R}}\otimes\tau_{\mathcal{R}})(t,s). \nonumber
\end{align}
Moreover,
$(A{}^{(n)}f_{\underline{t}(n),\underline{s}(n)}A)^{-\norm{.}_{2}}$
is orthogonal to
$(A{}^{(n)}f_{\underline{t}^{\prime}(n),\underline{s}^{\prime}(n)}A)^{-\norm{.}_{2}}$
whenever $\underline{t}(n)\neq \underline{s}(n)$,
$\underline{t}^{\prime}(n)\neq \underline{s}^{\prime}(n)$ and
$(\underline{t}(n),\underline{s}(n))\neq
(\underline{t}^{\prime}(n),\underline{s}^{\prime}(n))$.
\end{lem}
\begin{proof}
For $a,b\in A$, using Eq. \eqref{decompose} write
\begin{align}
\nonumber &a=\underset{\underline{q}(n)}\oplus
{}^{(n)}f_{\underline{q}(n)}\otimes a_{\underline{q}(n)} \text{ and
} b=\underset{\underline{p}(n)}\oplus
{}^{(n)}f_{\underline{p}(n)}\otimes b_{\underline{p}(n)} \nonumber
\end{align}
for $a_{\underline{q}(n)}\in A_{\infty , n+1}^{\underline{q}(n)}$,
and  $b_{\underline{p}(n)}\in A_{\infty , n+1}^{\underline{p}(n)}$.
By direct multiplication, we get
\begin{align}
\nonumber &a({}^{(n)}f_{\underline{t}(n),\text{
}\underline{s}(n)}\otimes 1_{\mathcal{R}_{n}})b
={}^{(n)}f_{\underline{t}(n),\text{ }\underline{s}(n)}\otimes
a_{\underline{t}(n)}b_{\underline{s}(n)}. \nonumber
\end{align}
\noindent Therefore $(i)$ follows from $(ii)$ of Lemma \ref{lemma:A}.
Moreover, for $a,b\in B$,
\begin{align}
\nonumber&\langle a({}^{(n)}f_{\underline{t}(n),\text{ }\underline{s}(n)}\otimes 1_{\mathcal{R}_{n}})b,{}^{(n)}f_{\underline{t}(n),\text{ }\underline{s}(n)}\otimes 1_{\mathcal{R}_{n}}\rangle_{\tau_{\mathcal{R}}}\\
\nonumber&=tr_{\prod _{i=1}^{n}k_{i}}({}^{(n)}f_{\underline{t}(n)})\tau_{\mathcal{R}_{n}}(a_{\underline{t}(n)}b_{\underline{s}(n)})\\
\nonumber&=\frac{1}{k_{1}k_{2}\cdots k_{n}}\tau_{\mathcal{R}_{n}}(a_{\underline{t}(n)}b_{\underline{s}(n)})\\
\nonumber&=k_{1}k_{2}\cdots k_{n}\tau_{\mathcal{R}}(a({}^{(n)}f_{\underline{t}(n)}\otimes1))\tau_{\mathcal{R}}(b({}^{(n)}f_{\underline{s}(n)}\otimes 1))\text{  (by orthogonality, Lemma \ref{lemma:A}} \text{ }(ii))\\
\nonumber&=k_{1}k_{2}\cdots k_{n}\int_{X}\int_{X}a(t)({}^{(n)}f_{\underline{t}(n)}\otimes 1)(t)b(s)({}^{(n)}f_{\underline{s}(n)}\otimes 1)(s)d(\tau_{\mathcal{R}}\otimes \tau_{\mathcal{R}})(t,s)\\
\nonumber&=k_{1}k_{2}\cdots k_{n}\int_{x_{t_{1}}^{(1)}\times\cdots\times x_{t_{n}}^{(n) }\times X^{(n)}}\int_{x_{s_{1}}^{(1)}\times\cdots\times x_{s_{n}}^{(n) }\times X^{(n)}}a_{\underline{t}(n)}(t)b_{\underline{s}(n)}(s)d(\tau_{\mathcal{R}}\otimes \tau_{\mathcal{R}})(t,s)\\
\nonumber &=k_{1}k_{2}\cdots k_{n}\int_{X\times
X}{}^{(n)}f_{\underline{t}(n)}(t){}^{(n)}f_{\underline{s}(n)}(s)a(t)b(s)d(\tau_{
\mathcal{R}}\otimes\tau_{\mathcal{R}})(t,s),
\end{align}
where the indicators of $(x_{t_{1}}^{(1)}\times\cdots\times
x_{t_{n}}^{(n)})\times X^{(n)}$ and $(x_{s_{1}}^{(1)}\times\cdots\times
x_{s_{n}}^{(n)})\times X^{(n)}$ corresponds to ${}^{(n)}f_{\underline{t}(n)}$ and
${}^{(n)}f_{\underline{s}(n)}$ respectively. This proves $(ii)$.
Clearly the final statement follows from $(i)$ and the fact that
${}^{(n)}f_{\underline{t}(n)}J{}^{(n)}f_{\underline{s}(n)}J$ and
${}^{(n)}f_{\underline{t}^{\prime}(n)}J{}^{(n)}f_{\underline{s}^{\prime}(n)}J$
are orthogonal projections in $L^{2}(\mathcal{R})$ if
$(\underline{t}(n),\underline{s}(n))\neq
(\underline{t}^{\prime}(n),\underline{s}^{\prime}(n))$.
\end{proof}

\begin{rem}\label{lemma:finalprod}
The following observation will be used in the next proof. On every occasion below, where we add direct integrals, Lemma 5.7 \cite{MR2261688} is invoked. For $t_{i}=s_{i}$, $1\leq i\leq n-1$ and $t_{n}\neq s_{n}$, the projection
$P_{(t_{1},\cdots,t_{n}),(s_{1},\cdots,s_{n})}^{(n)}\in \A^{\prime}$ and hence is
in $\A$, as $\A$ is maximal abelian in
$\mathbf{B}(L^{2}(\mathcal{R}))$. Therefore,
$P_{(t_{1},\cdots,t_{n}),(s_{1},\cdots,s_{n})}^{(n)}$ is decomposable
$($see Ch. 14 \cite{MR1468230}$)$. Denote
\begin{align}
\nonumber &E_{(\cdot,t_{n}),(\cdot,s_{n})}
=(x_{t_{1}}^{(1)}\times\cdots\times x_{t_{n-1}}^{(n-1)}\times
x_{t_{n}}^{(n) }\times X^{(n)})\times
(x_{t_{1}}^{(1)}\times\cdots\times x_{t_{n-1}}^{(n-1)}\times
x_{s_{n}}^{(n) }\times X^{(n)}).
\end{align}
>From Lemma \ref{lemma:prodm}, it follows that the range $P_{(t_{1},\cdots,t_{n}),(s_{1},\cdots,s_{n})}^{(n)}(L^{2}(\mathcal{R}))$ is the direct integral of complex numbers
over the set $E_{(\cdot,t_{n}),(\cdot,s_{n})}$ with respect to $\tau_{\mathcal{R}}\otimes \tau_{\mathcal{R}}$,
and, $\A P_{(t_{1},\cdots,t_{n}),(s_{1},\cdots,s_{n})}^{(n)}$ is the diagonalizable
algebra with respect to this decomposition. For
$\underline{t}(n-1)\neq \underline{t}^{\prime}(n-1)$, the direct
integrals of
$P_{(\underline{t}(n-1),t_{n}),(\underline{t}(n-1),s_{n})}^{(n)}$
and
$P_{(\underline{t}^{\prime}(n-1),t_{n}^{\prime}),(\underline{t}^{\prime}(n-1),s_{n}^{\prime})}^{(n)}$
with $t_{n}\neq s_{n}$ and $t_{n}^{\prime}\neq s_{n}^{\prime}$ rest
over disjoint subsets of $X\times X$. Therefore, the range of $P^{(n)}=\underset{t_{1}=s_{1},\cdots, t_{n-1}=s_{n-1},t_{n}\neq s_{n}}{\underset{\underline{t}(n),\text{ }\underline{s}(n):}\sum}P_{(t_{1},\cdots,t_{n}),(s_{1},\cdots,s_{n})}^{(n)}
$ is the direct integral of complex numbers with respect to $\tau_{\mathcal{R}}\otimes \tau_{\mathcal{R}}$ over the set
$E_{n}=\cup_{t_{1}=1}^{k_{1}}\cdots
\cup_{t_{n-1}=1}^{k_{n-1}}\cup_{t_{n}\neq
s_{n}=1}^{k_{n}}E_{(\cdot,t_{n}),(\cdot,s_{n})}$,
and associated statements about diagonalizability of $\A P^{(n)}$ hold. It is important to note that $E_{n}\cap E_{m}=\emptyset$ for all $n\neq m$.
\end{rem}

\indent Let $c_{n}=\overset{n}{\underset{r=1}\prod}k_{r}$ for
$n\geq 1$ and $c_{0}=1$.
\begin{prop}\label{mmobtained}
The vector
$\overset{\infty}{\underset
{n=1}\sum}\text{ }\underset{{\underline{t}(n)}}\sum\frac{1}{\sqrt{c_{n}}}{}^{(n)}f_{(\cdot,t_{n}),(\cdot,s_{n})}$
is a cyclic vector of $\A(1-e_{A})$ and
\begin{align}
\nonumber(1-e_{A})(L^{2}(\mathcal{R}))\cong \int_{X\times X}^
{\oplus}\C_{t,s}d(\tau_{\mathcal{R}}\otimes
\tau_{\mathcal{R}})(t,s), \text{ where }\C_{t,s}=\C.
\end{align}
Moreover, $\A(1-e_{A})$ is the algebra of diagonalizable operators
with respect to this decomposition.
\end{prop}

\begin{proof}
Fix $n\in \mathbb{N}$. For each $1 \leq t_{i} \leq k_{i}$, $1 \leq i
\leq n-1$, and $1 \leq t_{n}\neq s_{n}\leq k_{n}$, working with
vectors
$\frac{1}{\sqrt{c_{n}}}{}^{(n)}f_{(\cdot,t_{n}),(\cdot,s_{n})}$, one
finds $($using Lemma \ref{lemma:prodm}$)$ a positive measure
$\eta_{(t_{1},\cdots,t_{n}),(s_{1},\cdots,s_{n})}^{(n)}$ supported on
$E_{(\cdot,t_{n}),(\cdot,s_{n})}$ such that
\begin{align}
\nonumber
d\eta_{(t_{1},\cdots,t_{n}),(s_{1},\cdots,s_{n})}^{(n)}={}^{(n)}f_{(t_{1},\cdots,t_{n})}\otimes
{}^{(n)}f_{(s_{1},\cdots,s_{n})}d(\tau_{\mathcal{R}}\otimes
\tau_{\mathcal{R}}).
\end{align}
\indent By making arguments similar to Rem.
\ref{lemma:finalprod}, for each $n$ find a positive measure
$\eta^{(n)}$ on $E_{n}$
such that
$\eta^{(n)}=\chi_{E_{n}}d(\tau_{\mathcal{R}}\otimes
\tau_{\mathcal{R}})$ and
\begin{align}
\nonumber
P^{(n)}(L^{2}(\mathcal{R}))&=\underset{t_{1}=s_{1},\cdots, t_{n-1}=s_{n-1},t_{n}\neq s_{n}}{\underset{\underline{t}(n),\text{ }\underline{s}(n):}\sum}P_{(t_{1},\cdots,t_{n}),(s_{1},\cdots,s_{n})}^{(n)}
(L^{2}(\mathcal{R}))
\cong{\underset{X\times X}\int}^{\oplus} \C_{t,s}d\eta^{(n)}(t,s),
\end{align}
where $\C_{t,s}=\C$ and $\A P^{(n)}$ is diagonalizable with respect
to this decomposition.
Note that
\begin{equation}
\eta^{(n)}(X\times
X)=\frac{c_{n-1}(k_{n}^{2}-k_{n})}{c_{n}^{2}}=\frac{1}{c_{n-1}}-\frac{1}{c_{n}}.
\end{equation}
\indent From Rem. \ref{lemma:finalprod}, note that the
measures $\eta^{(n)}$ are supported on disjoint sets. Hence by
Lemma 5.7 \cite{MR2261688},
\begin{align}\label{final_direct_int}
(1-e_{A})(L^{2}(\mathcal{R}))\cong \int_{X\times X}^
{\oplus}\C_{t,s}d\eta(t,s), \text{ where }\C_{t,s}=\C, \text{ }
\eta=\overset{\infty}{\underset{n=1}\sum}\eta^{(n)}.
\end{align}
Moreover, $\A (1-e_{A})$ is diagonalizable with respect to the
decomposition in Eq. \eqref{final_direct_int}. Clearly,
\begin{equation}
\nonumber \eta(X\times X)=\lim_{N\to \infty}\sum
_{n=1}^{N}\eta^{(n)}(X\times X)=\lim_{N\to \infty}\frac{1}{c_{0}}- \frac{1}{c_{N}}=1.
\end{equation}
\indent Finally, $\eta=\tau_{\mathcal{R}}\otimes
\tau_{\mathcal{R}}$. Indeed, for $a$, $b$ $\in$ $C(X)$,
\begin{align}
\nonumber &\int_{X\times
X}a(t)b(s)d\eta(t,s)\\
\nonumber =&\overset{\infty}{\underset{n=1}\sum}\int_{X\times X}a(t)b(s)d\eta^{(n)}(t,s)\\
\nonumber
=&\overset{\infty}{\underset{n=1}\sum}\underset{t_{1}=s_{1},\cdots, t_{n-1}=s_{n-1},t_{n}\neq s_{n}}{\underset{\underline{t}(n),\text{ }\underline{s}(n):}\sum}\int_{X\times
X}a(t)b(s){}^{(n)}f_{(t_{1},\cdots,t_{n})}(t){}^{(n)}f_{(s_{1},\cdots,s_{n})}(s)d(\tau_{\mathcal{R}}\otimes\tau_{\mathcal{R}})(t,s).
\end{align}
But
$\overset{N}{\underset{n=1}\sum}\underset{t_{1}=s_{1},\cdots, t_{n-1}=s_{n-1},t_{n}\neq s_{n}}{\underset{\underline{t}(n),\text{ }\underline{s}(n):}\sum}{}^{(n)}f_{(t_{1},\cdots,t_{n})}(t){}^{(n)}f_{(s_{1},\cdots,s_{n})}(s)\uparrow
\chi_{\Delta(X)^{c}}$ pointwise $\tau_{\mathcal{R}}\otimes
\tau_{\mathcal{R}}$ almost everywhere. Use dominated convergence
theorem and the fact $(\tau_{\mathcal{R}}\otimes
\tau_{\mathcal{R}})(\Delta (X))= 0$ to conclude
$\eta=\tau_{\mathcal{R}}\otimes \tau_{\mathcal{R}}$. This completes
the proof.
\end{proof}

For $A$, the operator $x=\overset{\infty}{\underset
{n=1}\sum}\text{ }\underset{{\underline{t}(n)}}\sum\frac{1}{\sqrt{c_{n}}}{}^{(n)}f_{(\cdot,t_{n}),(\cdot,s_{n})}$ gives rise to a choice of a vector in $(iii)$ of Thm. \ref{sum_cond}. In order to get an appropriate vector one has to apply an appropriate transformation between the Cantor set and $[0,1]$, which will induce a unitary in $\mathbf{B}(L^{2}(\mathcal{R}))$ preserving the bimodule structure. Since the \emph{Puk{\'a}nszky invariant} of every \emph{Tauer masa} is $\{1\}$, we have computed the \emph{measure-multiplicity}
\emph{invariant} of $A$. Note that $AxA$ is dense in $L^{2}(\mathcal{R})\ominus L^{2}(A)$. For $a\in A$ and any orthonormal basis $\{v_{n}\}_{n=1}^{\infty}\subset A$ of $L^{2}(A)$, one has $\sum_{n}\norm{\mathbb{E}_{A}(xav_{n}x^{*})}_{2}^{2}=\sum_{n}\int_{X}\abs{\eta_{x}^{t}(1\otimes av_{n})}^{2}d\tau_{\mathcal{R}}(t)=\norm{a}_{2}^{2}$, as $\eta_{x}=\tau_{\mathcal{R}}\otimes \tau_\mathcal{R}$ $($see Lemma 3.6 \cite{MR11932708}$)$.
This shows that the Tauer masa above satisfy conditions $(i)$ and $(ii)$ of Thm. \ref{sum_cond} with $S=AxA$. The above Tauer masa was denoted by $A(0)$
in \cite{MR2302742}. There is a Tauer masa of exactly opposite flavor,
which we call the \emph{alternating Tauer masa}.

\subsection{Alternating Tauer Masa}

$ $\\\\
\indent The \emph{alternating Tauer masa} $A(1)$ is a singular Tauer
masa in the hyperfinite $\rm{II}_{1}$ factor $\mathcal{R}$,
constructed by White and Sinclair \cite{MR2302742}. It contains
nontrivial centralizing sequences of $\mathcal{R}$. In fact, its
$\Gamma$-invariant is $1$. This masa will play a role in \S 5. In \S 4, we will
describe its \emph{left-right-measure}.\\
\indent The chain for this masa is exactly similar to the masa of
the product class described before. Let
$A(1)_{1}=D_{2}(\mathbb{C})\subset \mathcal{M}_{1}$ be the diagonal
masa. Having constructed $A(1)_{n}\subset \mathcal{N}_{n}$, one
constructs $A(1)_{n+1}$ as,
\begin{equation}
A(1)_{n+1}=\begin{cases}
       A(1)_{n}\otimes {}^{(n+1)}D_{n+1} \text{, }n \text{ even, } {}^{(n+1)}D_{n+1} \text{ is the diagonal masa in }\mathcal{M}_{k_{n+1}}(\mathbb{C}),\\
       \underset{\underline{t}(n)}\bigoplus {}^{(n)}f_{\underline{t}(n)} \otimes {}^{(n+1)}D^{{\underline{t}(n)}}\text{, }n \text{ odd
       }, {}^{(n+1)}D^{{\underline{t}(n)}} \text{ pairwise
       orthogonal in }\mathcal{M}_{k_{n+1}}(\mathbb{C}).
\end{cases}
\end{equation}
\indent We will prove that the \emph{left-right-measure} of $A(1)$
is singular with respect to the product measure. Having understood
the \emph{left-right-measures} of $A(0)$ and $A(1)$, we can describe
the same for the entire path of masas
exhibited in \cite{MR2302742}.\\

\section{$\Gamma$ and Non $\Gamma$ Masas}

In this section, we study properties of
\emph{left-right-measures} of masas that possess nontrivial
centralizing sequences of the factor. We also study properties of
\emph{left-right-measures} that prevent a masa to contain
nontrivial centralizing sequences. This section contains partial
answers. Some results in this section can be proved by bringing in
the notion of strongly mixing masas \cite{MR2465603}. To keep this paper in
a reasonable size, we postpone the notion of strong mixing to a future
paper.

\begin{defn}
A centralizing sequence in a $\rm{II}_{1}$ factor $\mathcal{M}$ is a
bounded sequence $\{x_{n}\}\subset \mathcal{M}$ such that
$\norm{x_{n}y-yx_{n}}_{2}\rightarrow 0$ as $n\rightarrow \infty$ for
all $y\in \mathcal{M}$. The centralizing sequence $\{x_{n}\}$ is
trivial, if there exists a sequence $\lambda_{n}\in \C$ such that
$\norm{x_{n}-\lambda_{n}}_{2}\rightarrow 0$ as $n\rightarrow\infty$.
\end{defn}

For a masa $A\subset \mathcal{M}$, the $\Gamma$ invariant of $A$ is
defined by
\begin{align}
\nonumber \Gamma(A)=\sup\{\tau(p):& p\in A \text{ is a projection
and }   Ap \text{ contains nontrivial centralizing sequences of
}p\mathcal{M}p\}.
\end{align}
It is immediate that $\Gamma(A)=\Gamma(\theta(A))$, where $\theta$
is an automorphism of $\mathcal{M}$ \cite{MR2302742}. If
$\Gamma(A)=0$, then $A$ is said to be totally non-$\Gamma$. We
continue to assume that $A=L^{\infty}([0,1],\lambda)$, where
$\lambda$ is the Lebesgue measure.

\begin{prop}\label{central_seq_th1}
Let $A\subset \mathcal{M}$ be a masa. Let the left-right-measure of
$A$ be $[(\lambda\otimes\lambda)+\mu]$, where
$\mu\perp\lambda\otimes \lambda$ and $\mu$ is finite. Then $A$
cannot contain non trivial centralizing sequences of $\mathcal{M}$.
Moreover, $\Gamma(A)=0$.
\end{prop}

\begin{proof}
Write $[0,1]\times [0,1]\setminus \Delta([0,1])=E\cup F$, where
$(\lambda\otimes \lambda)(E)=0$ and $\mu(F)=0$. There exists a
nonzero vector $\zeta\in L^{2}(\mathcal{M})\ominus L^{2}(A)$ such
that for $a,b \in C[0,1]$,
\begin{align}
\nonumber \eta_{\zeta}(a\otimes b)=\lambda(a)\lambda(b).
\end{align}
The direct integral of $\zeta$ is supported on $F$.\\
\indent If possible, let $\{a_{n}\}\subset A$ be a non trivial
centralizing sequence. By
making a density argument, we can assume that $a_{n}=a_{n}^{*}\in
C[0,1]$ and $\tau(a_{n})=0$ for all $n$. Also assume that $\underset{n}\limsup
\norm{a_{n}}_{2}=\alpha
>0$. A triangle inequality argument shows that
$\norm{a_{n}\zeta-\zeta a_{n}}_{2}\rightarrow 0$ as
$n\rightarrow\infty$. However,
\begin{align}\label{central_seq_not_going_to_zero}
\norm{a_{n}\zeta-\zeta a_{n}}_{2}^{2}=&\langle a_{n}\zeta,a_{n}\zeta\rangle-\langle\zeta a_{n},a_{n}\zeta\rangle-\langle a_{n}\zeta,\zeta a_{n}\rangle+\langle\zeta a_{n},\zeta a_{n}\rangle\\
\nonumber=&2\lambda(a_{n}^{*}a_{n}).
\end{align}
Eq. \eqref{central_seq_not_going_to_zero} shows that
$\norm{a_{n}\zeta-\zeta a_{n}}_{2}^{2}\not\rightarrow 0$ as
$n\rightarrow \infty$, which is a contradiction.\\
\indent The last statement follows from the above argument by
considering compressions of $\mathcal{M}$ by projections in $A$,
because, for any nonzero projection $p\in A$, identifying $p$ as the
indicator of a measurable set $E_{p}$, it follows that the
\emph{left-right-measure} of the inclusion $Ap\subset p\mathcal{M}p$
will be the class of the restriction of $\lambda\otimes \lambda
+\mu$ to $E_{p}\times E_{p}$.
\end{proof}

The next result is a generalization of Prop. \ref{central_seq_th1}.
We skip its proof, as the proof is similar to the proof of Prop.
\ref{central_seq_th1}.

\begin{prop}\label{no_central_on_square}
Let $A\subset \mathcal{M}$ be a masa. Let the left-right-measure of
$A$ restricted to the projection $pJqJ$ contain the product measure as a
summand, where $p$ and $q$ are nonzero projections in $A$.
Then: \\
$(i)$ $\Gamma(A)< 1$.\\
$(ii)$ If $r\geq p,q$ is any projection in $A$, then $Ar$ cannot
contain nontrivial centralizing sequences of $r\mathcal{M}r$.
\end{prop}

\begin{prop}\label{central_seq_th2}
Let $A\subset \mathcal{M}$ be a masa. Let the left-right-measure of
$A$ be $[\nu+\mu]$, where $\mu\perp \lambda\otimes\lambda$, $\nu\ll
\lambda\otimes\lambda$, $\nu$ and $\mu$ are finite and $\nu\neq 0$. Then
$A$ cannot contain any centralizing sequence of $\mathcal{M}$
consisting of weakly null unitaries.
\end{prop}

\begin{proof}
Without loss of generality, we can assume that
$f=\frac{d\nu}{d(\lambda\otimes\lambda)}\in
L^{2}(\lambda\otimes\lambda)$. Write $[0,1]\times [0,1]\setminus
\Delta([0,1])=E\cup F$, where $\nu(E)=0$ and $\mu(F)=0$. There
exists a nonzero vector $\zeta_{0}\in L^{2}(\mathcal{M})\ominus
L^{2}(A)$ such that for $a,b \in C[0,1]$,
\begin{align}
\nonumber \eta_{\zeta_{0}}(a\otimes b)=\int_{[0,1]\times
[0,1]}a(t)b(s)f(t,s)d\lambda(t)d\lambda(s).
\end{align}
The direct integral of $\zeta_{0}$ is supported on $F$. Arguing as in the
proof of Thm. \ref{measure_to_sum_all_puk}, we conclude that
$\mathbb{E}_{A}(\zeta_{0} b\zeta_{0}^{*})\in L^{2}(A)$ for all $b\in
C[0,1]$ and $\sum_{k\in \mathbb{Z}}\norm{\mathbb{E}_{A}(\zeta_{0}
v^{k}\zeta_{0}^{*})}_{2}^{2}<\infty$, where $v\in A$ is the Haar
unitary
generator corresponding to the function $t\mapsto e^{2\pi it}$.\\
\indent Suppose to the contrary, there is a sequence
$\{a_{n}\}\subset C[0,1]\subset A$ of weakly null unitaries that
centralize $\mathcal{M}$. Given $\epsilon>0$, choose $k_{0}\in
\mathbb{N}$ such that $\sum_{\abs{k}\geq
k_{0}}\norm{\mathbb{E}_{A}(\zeta_{0}v^{k}\zeta_{0}^{*})}_{2}^{2}<\epsilon^{2}$.
Therefore on one hand,
\begin{align}
\nonumber \norm{\mathbb{E}_{A}(\zeta_{0}a_{n}\zeta_{0}^{*})}_{1}&=\norm{\sum_{k\in \mathbb{Z}}\tau(a_{n}v^{-k})\mathbb{E}_{A}(\zeta_{0}v^{k}\zeta_{0}^{*})}_{1}\\
\nonumber &\leq\norm{\sum_{\abs{k}<k_{0}}\tau(a_{n}v^{-k})\mathbb{E}_{A}(\zeta_{0}v^{k}\zeta_{0}^{*})}_{1}+\norm{\sum_{\abs{k}\geq k_{0}}\tau(a_{n}v^{-k})\mathbb{E}_{A}(\zeta_{0}v^{k}\zeta_{0}^{*})}_{1}\\
\nonumber &\leq\sum_{\abs{k}<k_{0}}\abs{\tau(a_{n}v^{-k})}\norm{\mathbb{E}_{A}(\zeta_{0}v^{k}\zeta_{0}^{*})}_{1}+\sum_{\abs{k}\geq k_{0}}\abs{\tau(a_{n}v^{-k})}\norm{\mathbb{E}_{A}(\zeta_{0}v^{k}\zeta_{0}^{*})}_{2}\\
\nonumber &\leq\sum_{\abs{k}<k_{0}}\abs{\tau(a_{n}v^{-k})}\norm{\mathbb{E}_{A}(\zeta_{0}v^{k}\zeta_{0}^{*})}_{1}+\left(\sum_{\abs{k}\geq k_{0}}\abs{\tau(a_{n}v^{-k})}^{2}\right)^{\frac{1}{2}}\left(\sum_{\abs{k}\geq k_{0}}\norm{\mathbb{E}_{A}(\zeta_{0}v^{k}\zeta_{0}^{*})}_{2}^{2}\right)^{\frac{1}{2}}\\
\nonumber
&\leq\sum_{\abs{k}<k_{0}}\abs{\tau(a_{n}v^{-k})}\norm{\mathbb{E}_{A}(\zeta_{0}v^{k}\zeta_{0}^{*})}_{1}+\epsilon.
\end{align}
Since $a_{n}\overset{w.o.t}\rightarrow 0$ and $\epsilon$ is
arbitrary, so
$\norm{\mathbb{E}_{A}(\zeta_{0}a_{n}\zeta_{0}^{*})}_{1}\rightarrow
0$. On the other hand, $\norm{a_{n}^{*}\zeta_{0}
a_{n}\zeta_{0}^{*}-\zeta_{0}\zeta_{0}^{*}}_{1}\rightarrow 0$ as
$n\rightarrow \infty$ and consequently
$\norm{a_{n}^{*}\mathbb{E}_{A}(\zeta_{0}
a_{n}\zeta_{0}^{*})-\mathbb{E}_{A}(\zeta_{0}\zeta_{0}^{*})}_{1}\rightarrow
0$. So $\norm{\mathbb{E}_{A}(\zeta_{0}
a_{n}\zeta_{0}^{*})}_{1}\rightarrow
\norm{\mathbb{E}_{A}(\zeta_{0}\zeta_{0}^{*})}_{1}$. This is a
contradiction as $\zeta_{0}$ is nonzero.
\end{proof}

\begin{cor}
The left-right-measure of $A(1)$ is singular with respect to the product class.
\end{cor}
\begin{proof}
By construction $A(1)$ contains a centralizing sequence of weakly null unitaries.
\end{proof}

\begin{cor}\label{McDuff_corollary}
Every strongly stable $($McDuff$)$ factor contains a
singular masa whose left-right-measure is singular with respect to
the product class.
\end{cor}
\begin{proof}
For existence of a singular masa in a $\rm{II}_{1}$ factor see \cite{MR693226}. The statement follows by tensoring any singular masa in the factor by the alternating Tauer masa in $\mathcal{R}$
$($see \cite{MR693226,MR999995}, Prop. 5.2 \cite{MR2261688} and Lemma 3.5 \cite{MR11932708}$)$.
\end{proof}

It is now natural to ask the following question. If $\mathcal{M}=M_{1}\overline{\otimes} M_{2}$, where both $M_{1},M_{2}$ are $\rm{II}_{1}$ factors then does $\mathcal{M}$ contain a $($singular$)$ masa whose \emph{left-right-measure} is singular with respect to the
product class?\\
$$ $$
\begin{defn}\label{define_rigid_measure}
A finite measure $\mu$ on $[0,1]$ $($or $S^{1})$ is called $\alpha$-\emph{rigid}
for $\abs{\alpha}=1$, if and only if, there is a subsequence
$\widehat{\mu}_{n_{k}}$ of $\widehat{\mu}_{n}=\int_{0}^{1}e^{-2\pi
int}d\mu(t)$ $($or $\widehat{\mu}_{n}=\int_{S^{1}}z^{-n}d\mu(z))$ that converges to $\alpha\mu([0,1])$ $($or $\alpha\mu(S^{1}))$ as $k\rightarrow
\infty$. A $1$-rigid measure is called rigid or a Dirichlet measure.
\end{defn}

\indent We now recall some properties of $\alpha$-rigid measures.
For details check Ch.7 \cite{MR1719722}. Let $\mu$ be a $\alpha$-rigid
measure on $[0,1]$. Any sequence $n_{k}$ along which
$\widehat{\mu}_{n_{k}}$ converges to $\alpha\mu([0,1])$ is said to
be a sequence associated with $\mu$. It is easy to see that, $\mu$
is $\alpha$-rigid, if and only if, the sequence of functions $[0,1]\ni
t\mapsto e^{-2\pi in_{k}t}$ converges to $\alpha$ in $\mu$-measure.
Thus $\nu$ is $\alpha$-rigid with associated sequence $n_{k}$ for
any $\nu\ll \mu$. So $\alpha$-rigidity is a property of equivalence
class of measures, and hence can be thought of as a property of
unitary operators, by considering appropriate Koopman operators.
Atomic measures are always rigid.\\

\indent To motivate what follows, we consider rigid m.p. transformations. Let $T$ be a m.p. automorphism of a standard probability space $(X,\mu)$. Let $U_{T}$ denote the associated Koopman operator on $L^{2}(X,\mu)$. The transformation $T$ is said to be rigid if $1\in \overline{\{U_{T}^{n}\}_{n\in \mathbb{Z}\setminus \{0\}}}^{s.o.t}$ \cite{MR1940356}.\\
\indent Assume further that $T$ is weakly mixing. Then $L(\mathbb{Z})\subset L^{\infty}(X,\mu)\rtimes_{T}\mathbb{Z}$ is a singular masa \cite{MR0268687} $($also see Thm. 2.1 \cite{MR1940356}$)$. Let $U_{T}^{n_{k}}\overset{s.o.t}\rightarrow 1$ as $k\rightarrow\infty$. A simple calculation shows that $L(\mathbb{Z})$ contains a centralizing sequence of the crossed product factor consisting of powers of the standard Haar unitary generator. It is not known whether this is always the case for $\Gamma$-masas. Let $\nu$ $($which is a measure on $\widehat{\mathbb{Z}}=S^{1})$ denote the maximal spectral type of the action $T$. Then there is a unit vector $f\in L^{2}(X,\mu)$ such that $\widehat{\nu}_{n}=\langle U_{T}^{n}f,f\rangle$ for all $n\in \mathbb{Z}$. It follows that $\nu$
is a Dirichlet measure. The relationship between the maximal spectral type of an action and the \emph{left-right measure}
of the associated masa appeared in Prop. 3.1 \cite{MR1940356}. Thus by general theory of $\alpha$-rigid measures $($see Ch. 7 \cite{MR1719722}$)$, it follows that for $\lambda$ almost all $t$ $(\lambda$ is Haar measure$)$, the measure $\tilde{\eta}^{t}$ is $\alpha$-rigid for all $\alpha\in S^{1}$.\\
\indent In the general case, when $A$ contains a nontrivial centralizing sequence of $\mathcal{M}$, one can choose a central sequence consisting of trigonometric polynomials without constant term. We do not know whether we can choose a central sequence of the form $t\mapsto e^{2\pi in_{k}t}$. In case we can, results analogous to the crossed product situation hold.

$$ $$
\indent Making appropriate changes to the proof of Lemma
\ref{identify_disintegrated_measure}, we get the following result.
Its proof uses basic facts about $L^{1}$ spaces associated to finite
von Neumann algebras. We omit its proof.
\begin{lem}\label{compute_l_1_norm}
Let $\zeta\in L^{2}(\mathcal{M})$ be such that
$\mathbb{E}_{A}(\zeta)=0$. Let $\eta_{\zeta}$ denote the measure on
$[0,1]\times [0,1]$ defined in Eq. \eqref{measure_from_kappa}. Let
$b,w\in C[0,1]$. Then
\begin{align}
\nonumber \norm{\mathbb{E}_{A}(b\zeta
w\zeta^{*})}_{1}=\int_{0}^{1}\abs{b(t)}\abs{{\eta_{\zeta}^{t}(1\otimes
w)}}d\lambda(t).
\end{align}
\end{lem}

\begin{thm}\label{cenral_seq_torigid}
Let $A\subset \mathcal{M}$ be a singular masa. Let $v\in A$ be a
Haar unitary generator of $A$. Suppose there exists a subsequence
$n_{k}$ $(n_{k}<n_{k+1}$ for all k$)$ such that for all $y\in
\mathcal{M}$,
\begin{align}
\nonumber \norm{v^{n_{k}}y-yv^{n_{k}}}_{2}\rightarrow 0 \text{ as }
k\rightarrow \infty.
\end{align}
Then the measure $\tilde{\eta}^{t}$ is $\beta$-rigid for all
$\beta\in S^{1}$, $\lambda$ almost all $t$, where
$[\eta]$ is the left-right-measure of $A$.
\end{thm}

\begin{proof}
Let $w$ be the Haar unitary generator of $A$ that corresponds to the function $[0,1]\ni
t\mapsto e^{2\pi it}$. The map from $L^{\infty}([0,1],\lambda)$ to itself, which sends
$v^{n}$ to $w^{n}$ for $n\in\mathbb{Z}$, implements a m.p. Borel isomorphism $T:[0,1]\mapsto [0,1]$. Then $T\times T$ implements
an unitary $U:L^{2}(\mathcal{M})\mapsto L^{2}(\mathcal{M})$, which preserves the structure of $L^{2}(\mathcal{M})$ as the
natural $A,A$-bimodule $($see Defn. \ref{mminv}$)$. Standard density arguments show that if
$\xi\in L^{2}(\mathcal{M})$, then
\begin{align}
\nonumber \norm{v^{n_{k}}\xi-\xi v^{n_{k}}}_{2}\rightarrow 0 \text{
as } k\rightarrow \infty.
\end{align}
So, we can assume $v=w$.\\
\indent We know that there is a nonzero vector $\zeta\in
L^{2}(\mathcal{M})\ominus L^{2}(A)$ such that $\eta=\eta_{\zeta}$.
Therefore $\norm{\mathbb{E}_{A}( v^{-n_{k}}\zeta v^{n_{k}}\zeta^{*}
)-\mathbb{E}_{A}(\zeta\zeta^{*})}_{1}\rightarrow 0$ as $k\rightarrow
\infty$. Consequently, using similar arguments that are needed to
prove Lemma \ref{compute_l_1_norm}, we have,
\begin{align}
\nonumber\norm{\mathbb{E}_{A}( v^{-n_{k}}\zeta v^{n_{k}}\zeta^{*}
)-\mathbb{E}_{A}(\zeta\zeta^{*})}_{1}=\int_{0}^{1}\abs{e^{-2\pi
in_{k}t}\eta^{t}(1\otimes v^{n_{k}})-\mathbb{E}_{A}(\zeta
\zeta^{*})(t)}d\lambda(t)\rightarrow 0
\end{align}
as $k\rightarrow \infty$. Hence, there exists a further subsequence
$n_{k_{l}}$ and a subset $E\subset [0,1]$ such that $\lambda(E)=0$,
and for $t\in E^{c}$,
\begin{align}\label{convae}
e^{-2\pi in_{k_{l}}t}\eta^{t}(1\otimes
v^{n_{k_{l}}})-\mathbb{E}_{A}(\zeta \zeta^{*})(t)\rightarrow 0
\text{ as }l\rightarrow \infty.
\end{align}
Note that $\mathbb{E}_{A}(\zeta
\zeta^{*})(t)=\tilde{\eta}^{t}([0,1])<\infty$ $($see Lemma \ref{identify_disintegrated_measure}$)$ almost everywhere $\lambda$.\\
\indent It is known that for almost every $\beta\in [0,1]$ $($with respect to $\lambda)$, the set of limit points of the sequence $e^{-2\pi i n_{k_{l}}\beta}$ contains a point of the form $e^{2\pi i\alpha}$ with $\alpha$ irrational $($see Ch. 7 \cite{MR1719722}$)$. Thus by enlarging the null set $E$ and renaming it to be $E$ again, we conclude that $e^{-2\pi
in_{k_{l_{m}}}t}\rightarrow e^{2\pi i\alpha_{t}}$ for $t\in E^{c}$, $\alpha_{t}$ irrational. The subsequence in the last statement depends on $t$. By a diagonal argument, it follows that for $t\in E^{c}$, the measure $\tilde{\eta}^{t}$ is $\beta$-rigid for all $\beta\in S^{1}$ $($see Ch. 7 \cite{MR1719722}$)$.
\end{proof}

\begin{rem}
Examples of singular masas in the hyperfinite $\rm{II}_{1}$ factor
can be constructed that satisfy the hypothesis of Thm.
\ref{cenral_seq_torigid}. There exist weakly mixing actions of a
stationary Gaussian process that has the desired properties $($check \S5
\cite{MR0291413}$)$. In fact, if $A$ is Cartan in $\mathcal{R}$, then there is a centralizing sequence in $A$ consisting of powers of some Haar unitary generator. This follows from Thm. 5.5 \cite{MR11932708}, Prop. 3.1 \cite{MR1940356}, Thm. 4 \cite{MR0291413} and \cite{MR662736}. For example, consider the Cartan masa in the hyperfinite $\rm{II}_{1}$ factor which arises from a irrational rotation along the direction of the group.
\end{rem}

\section{Examples of Singular Masas in the Free Group Factors}

In this section, we show that given any subset $S$ of $\mathbb{N}$,
there are uncountably many pairwise non conjugate singular masas in
$L(\mathbb{F}_{k})$, $k\geq 2$, with Puk{\'a}nszky invariant $S\cup
\{\infty\}$. All examples exhibited in this section are
constructed from examples appearing in \cite{MR2261688,MR2163938}.
For any masa $A$ considered in this section, we
assume $A=L^{\infty}([0,1],\lambda)$, where $\lambda$ is the
Lebesgue measure. If $A\subset \mathcal{M}$ is a masa and $[\eta]$
is its \emph{left-right-measure}, then we will most of the time
assume that $\eta(\Delta[0,1])=0$. There are few exceptions to this
assumption in this section, in which case we shall notify accordingly. The next
two corollaries are direct applications of results in
\cite{MR2261688} $($see Lemma 3.1, Thm. 3.2, Lemma 5.7 and Prop. 5.10 of \cite{MR2261688}$)$. The singularity of a masa $A\subset M\subset M*N$ $(M,N$ are diffuse$)$ as deduced in this section from results in \cite{MR11932708}, can also be deduced from Thm. 2.3 \cite{MR2261688} or \cite{MR703810}.

\begin{cor}\label{freeley_complemented_1}
Let $k\in \mathbb{N}_{\infty}$ and $k\geq 2$. Let $A\subset
L(\mathbb{F}_{k})$ be a masa. If $A$ is freely complemented then
$Puk(A)=\{\infty\}$ and its left-right-measure is the class of
product measure. In particular, $A$ is singular.
\end{cor}
\begin{proof}
Follows directly from Lemma $5.7$ and Prop. $5.10$ \cite{MR2261688}.
Singularity follows from Cor. 3.2 in \cite{MR815434}.
\end{proof}

\begin{cor}\label{freeley_complemented_2}
Let $k\in \mathbb{N}_{\infty}$ and $k\geq 2$. Let $A\subset
L(\mathbb{F}_{k})$ be a masa. Let $A\varsubsetneq B\varsubsetneq
L(\mathbb{F}_{k})$, where $B$ is a subalgebra and $B$ is freely
complemented.\\
$(i)$ If the left-right-measure $[\eta_{B}]$ of the inclusion
$A\subset B$ is singular with respect to $\lambda\otimes\lambda$,
then $Puk_{L(\mathbb{F}_{k})}(A)=Puk_{B}(A)\cup \{\infty\}$ and the
left-right-measure of the inclusion $A\subset L(\mathbb{F}_{k})$
is $[\eta_{B}+\lambda\otimes \lambda]$.\\
$(ii)$ If the left-right-measure $[\eta_{B}]$ of the inclusion
$A\subset B$ is $[\lambda\otimes\lambda]$, then
$Puk_{L(\mathbb{F}_{k})}(A)=\{\infty\}$ and the left-right-measure
of the inclusion $A\subset L(\mathbb{F}_{k})$ is
$[\lambda\otimes\lambda]$.
\end{cor}
\begin{proof}
Follows from Lemma $5.7$ and Prop $5.10$ \cite{MR2261688}.
\end{proof}

Let $T$ be a nonempty subset of $\mathbb{N}$. Let $T=\{n_{k}\}$ with
$n_{1}<n_{2}<\cdots$. Define
\begin{align}
\mathbb{P}_{T}=\left\{ \alpha=\{\alpha_{n_{k}}\}_{k=1}^{\abs{T}}:
\alpha_{n_{k}}>\alpha_{n_{k+1}}, 0<\alpha_{n_{k}}<1 \text{ for all }
k,\sum_{k=1}^{\abs{T}}\alpha_{n_{k}}=1\right\}.
\end{align}
For $\alpha,\beta \in \mathbb{P}_{T}$, we say $\alpha\neq \beta$ if
$\alpha_{n_{k}}\neq \beta_{n_{k}}$ for some $k$.

\begin{thm}\label{unctblymany}
Let $B\subset \mathcal{R}$ be a singular masa such that the
left-right-measure $[\eta_{\mathcal{R}}]$ of the inclusion $B\subset
\mathcal{R}$ is singular with respect to $\lambda\otimes \lambda$.
For each $\alpha\in \mathbb{P}_{\mathbb{N}}$, there exists a
singular masa $B_{\alpha}\subset L(\mathbb{F}_{2})$ with
$Puk_{L(\mathbb{F}_{2})}(B_{\alpha})=Puk_{\mathcal{R}}(B)\cup
\{\infty\}$. If $\alpha\neq \beta$ are any two elements of
$\mathbb{P}_{\mathbb{N}}$, then $B_{\alpha}$ and $B_{\beta}$ are not
conjugate.
\end{thm}

\begin{proof}
Fix $\alpha\in \mathbb{P}_{\mathbb{N}}$. Let
$\mathcal{R}_{\alpha}=\oplus_{n=1}^{\infty}\mathcal{R}$. Equip
$\mathcal{R}_{\alpha}$ with the faithful trace
\begin{align}
\nonumber
\tau_{\mathcal{R}_{\alpha}}(\cdot)=\sum_{n=1}^{\infty}\alpha_{n}\tau_{\mathcal{R}}(\cdot),
\end{align}
where $\tau_{\mathcal{R}}$ denotes the unique normal tracial state of
$\mathcal{R}$.\\
\indent Then $B_{\alpha}=\oplus_{n=1}^{\infty}B$ is a singular masa
in the hyperfinite algebra $\mathcal{R}_{\alpha}$ and the latter is
separable. The projections $(0\oplus \cdots \oplus 0\oplus 1\oplus
0\oplus \cdots)$, where $1$ appears at the $n$-th coordinate, is a
central projection $p_{n}$ of $\mathcal{R}_{\alpha}$ and it belongs
to $B_{\alpha}$. The projections $p_{n}$ correspond to the indicator function of
measurable subsets $F_{n}\subset ([0,1],\lambda)$ respectively, so that
$F_{n}\cap F_{m}$ is a set of $\lambda$ measure $0$ for all $n\neq
m$. By applying appropriate transformations, the
\emph{left-right-measure} of $B\subset \mathcal{R}$ can be
transported to each $F_{n}\times F_{n}$, which we denote by
$[\eta_{n}]$. We also assume $\eta_{n}(F_{n}\times F_{n})=1$ for all
$n$. Note that by factoriality of $\mathcal{R}$ it follows that $\eta_{n}(E\times F)>0$ for all measurable rectangles $E\times F\subset F_{n}\times F_{n}$ such that $\lambda(E)>0,\text{ }\lambda(F)>0$ $($Lemma 2.9 \cite{MR11932708}$)$.\\
\indent Consider
$(\mathcal{M},\tau_{\mathcal{M}})=(\mathcal{R}_{\alpha},\tau_{\mathcal{R}_{\alpha}})*(\mathcal{R},\tau_{\mathcal{R}})$.
Then $\mathcal{M}$ is isomorphic to $L(\mathbb{F}_{2})$ by a well
known theorem of Dykema \cite{MR1201693}. Then $B_{\alpha}\subset
L(\mathbb{F}_{2})$ is a singular masa by Thm. 2.3 \cite{MR2261688}.
The \emph{left-right-measure} of the inclusion $B_{\alpha}\subset
L(\mathbb{F}_{2})$ is
\begin{align}
\nonumber[\lambda\otimes\lambda +
\sum_{n=1}^{\infty}\frac{1}{2^{n}}\eta_{n}]
\end{align}
and $Puk_{L(\mathbb{F}_{2})}(B_{\alpha})=\cup_{n}
Puk_{(\mathcal{R}*\mathcal{R})}(B)\cup
\{\infty\}=Puk_{\mathcal{R}}(B)\cup \{\infty\}$ from Cor.
\ref{freeley_complemented_2} and Thm. 3.2
\cite{MR2261688}.\\
\indent Since automorphisms of $\rm{II}_{1}$ factors preserve the
trace and orthogonal projections, the non conjugacy of $B_{\alpha}$
and
$B_{\beta}$ for $\alpha\neq \beta$ follows by considering the \emph{left-right-measures}. Indeed, if $B_{\beta}=\phi(B_{\alpha})$ for some automorphism $\phi$ of $L(\mathbb{F}_{2})$ and $\alpha,\beta\in \mathbb{P}_{\mathbb{N}}$, then $B_{\alpha}$ and $B_{\beta}$ would have identical bimodule structure.
Therefore, there is a Borel isomorphism $\tilde{\phi}:[0,1]\mapsto [0,1]$ such that $\tilde{\phi}_{*}\lambda=\lambda$ and $(\tilde{\phi}\times \tilde{\phi})_{*}[\eta_{B_{\alpha}}]=[\eta_{B_{\beta}}]$, where $\eta_{B_{\alpha}}$, $\eta_{B_{\beta}}$ denote the \emph{left-right-measures} of $B_{\alpha}$ and $B_{\beta}$ respectively. \\
\indent Let $F_{n},E_{n}, n=1,2,\cdots$, be the measurable partitions of $[0,1]$ associated to the \emph{left-right-measures} of $B_{\alpha}$ and $B_{\beta}$ $($as described above$)$ respectively. Clearly, the class of the singular part of $\eta_{B_{\alpha}}$ will be pushed forward to the class of the singular part of $\eta_{B_{\beta}}$. If we let $\chi_{E_{n}^{\prime}}=\phi(\chi_{F_{n}})$, then $E_{n}^{\prime}=E_{k_{n}}$ mod $\lambda$ for some $k_{n}$. But the $\lambda$-measure of $F_{n}$ and hence $E_{k_{n}}$ are strictly decreasing as $n$ increases. Thus $k_{n}=n$ for all $n$. This completes the proof.
\end{proof}

Now we construct non conjugate singular masas in the free group
factors which have the same multiplicity. We will give a case by case argument.\\
\noindent \textbf{Case:} $\{1,\infty\}$: In Thm. \ref{unctblymany},
let $B$ be the alternating Tauer masa $A(1)$.\\
\noindent \textbf{Case:} $\{1,n,\infty\}$, $n\neq 1$: Consider the
matrix groups
\begin{align}\label{groups}
G_{n}=\left\{(\begin{smallmatrix} f&x\\0 &
1\end{smallmatrix}\bigr)\mid f\in P_{n},x\in
\mathbb{Q}\right\},\text{  } H_{n}=\left\{(\begin{smallmatrix}
f&0\\0 & 1\end{smallmatrix}\bigr)\mid f\in P_{n}\right\}\subset
G_{n}, \text{ where }
\end{align}
\begin{align}
\nonumber &P_{\infty}=\left\{\frac{p}{q}\mid p,q\in \mathbb{Z}, p,q
\text{ odd }\right\} \text{and }P_{n}=\left\{ f2^{kn}\mid f\in
P_{\infty}, k\in \mathbb{Z}\right\} \text{ if }n<\infty,
\end{align}
are subgroups of the multiplicative group of nonzero rational
numbers. Then $L(G_{n})$ is the hyperfinite $\rm{II}_{1}$ factor
$\mathcal{R}$ and $L(H_{n})\subset L(G_{n})$ is a singular masa with
\emph{Puk\'{a}nszky invariant} $\{n\}$. The
\emph{left-right-measure} of the inclusion $L(H_{n})\subset
L(G_{n})$ is the class of product Haar measure
$\lambda_{\widehat{H}_{n}}\otimes\lambda_{\widehat{H}_{n}}$ on
$\widehat{H}_{n}\times \widehat{H}_{n}$, where $\widehat{H}_{n}$
denotes the character group of $H_{n}$ (see Example 5.1 \cite{MR2163938} and Example 6.2 \cite{MR2261688}$)$.\\
\indent As $\mathcal{R}\overline{\otimes}\mathcal{R}\cong
\mathcal{R}$, so $L(H_{n})\overline{\otimes} A(1)$ is a singular
masa in $\mathcal{R}$ from \cite{MR999995}
$($see also \cite{MR999997} and Thm. 5.15 \cite{MR11932708}$)$. Let $A(1)=L^{\infty}([0,1],\lambda)$. The
\emph{Puk\'{a}nszky invariant} of the inclusion
$L(H_{n})\overline{\otimes} A(1)\subset \mathcal{R}$ is $\{1,n\}$
from Thm. $2.1$ \cite{MR2163938}. The \emph{left-right-measure}
of the inclusion $L(H_{n})\overline{\otimes} A(1)\subset
\mathcal{R}$ is the class of
\begin{align}
\nonumber\lambda_{\widehat{H}_{n}}\otimes\lambda_{\widehat{H}_{n}}\otimes
\eta+\tilde{\Delta}_{*}\lambda_{\widehat{H}_{n}}\otimes \eta+
\lambda_{\widehat{H}_{n}}\otimes\lambda_{\widehat{H}_{n}}\otimes\tilde{\Delta}_{*}\lambda,
\end{align}
on $\widehat{H}_{n}\times \widehat{H}_{n}\times [0,1]\times [0,1]$,
where $[\eta]$ is the \emph{left-right-measure} of the
\emph{alternating Tauer masa} restricted to the off diagonal and
$\tilde{\Delta}$ is the map, that maps a set to its square by sending
$x\mapsto (x,x)$ $($see Prop. 5.2 \cite{MR2261688}$)$. \emph{In this case, we need to specify the measures
on the diagonals as they are necessary.} Given $\alpha\in
\mathbb{P}_{\mathbb{N}}$, replace the role of $B$ in Thm.
\ref{unctblymany} by $L(H_{n})\overline{\otimes} A(1)$ to construct
a masa $A_{\alpha,n}\subset L(\mathbb{F}_{2})$.\\
\noindent \textbf{Case:} $\{n,\infty\}$, $n\neq 1$: Let
$H_{n}\subset G_{n}$ and $H_{\infty}\subset G_{\infty}$ be as in the
previous case. Then $L(H_{n}\times H_{\infty})$ is a singular masa
in $L(G_{n}\times G_{\infty})$ whose \emph{measure-multiplicity
 invariant} is the equivalence class of
$$(\widehat{H}_{n}\times \widehat{H}_{\infty}, \lambda_{\widehat{H}_{n}}\otimes \lambda_{\widehat{H}_{\infty}},[\eta],m),$$ where $\eta$
is the sum of\\
$(i)\text{   }$ Haar measure on $(\widehat{H}_{n}\times
\widehat{H}_{\infty})\times(\widehat{H}_{n}\times
\widehat{H}_{\infty})$;\\
$(ii)$ Haar measure on the subgroup
\begin{align}
\nonumber D_{\infty}=\{(\alpha,\beta_{1},\alpha,\beta_{2})\mid
\alpha\in\widehat{H}_{n}, \beta_{1},\beta_{2}\in
\widehat{H}_{\infty} \};
\end{align}
$(iii)$ Haar measure on the subgroup
\begin{align}
\nonumber D_{n}=\{(\alpha_{1},\beta,\alpha_{2},\beta)\mid
\alpha_{1},\alpha_{2}\in\widehat{H}_{n}, \beta\in
\widehat{H}_{\infty} \};
\end{align}
and where the multiplicity function on the off-diagonal is given by
\begin{equation}
\nonumber m(\gamma)=\begin{cases}
         n, &\text{ }\gamma\in D_{n}\setminus \Delta(\widehat{H}_{n}\times
         \widehat{H}_{\infty}),\\
         \infty, &\text{ otherwise}.
         \end{cases}
\end{equation}
\indent This was calculated in \S6, \cite{MR2261688}. Note that $\eta$
contain singular summands, singular with respect to product Haar
measure off the diagonal $\Delta(\widehat{H}_{n}\times
\widehat{H}_{\infty})$. For each $\alpha\in
\mathbb{P}_{\mathbb{N}}$, make a construction analogous to Thm.
\ref{unctblymany} with $B$ replaced by $L(H_{n})\overline{\otimes}
L(H_{\infty})$, to construct a masa $A_{\alpha,n}\subset
L(\mathbb{F}_{2})$. Note that
$Puk_{L(\mathbb{F}_{2})}(A_{\alpha,n})$ $=\{n,\infty\}$ from Thm.
3.2 and Lemma 5.7 \cite{MR2261688}. The \emph{left-right-measure} of
the inclusion $A_{\alpha,n}\subset L(\mathbb{F}_{2})$ is of the same
form as discussed in the previous cases. Use Lemma 3.6, Thm. 3.8 of
\cite{{MR11932708}} to decide non conjugacy of
$A_{\alpha,n},A_{\beta,n}$ whenever $\alpha\neq \beta \in
\mathbb{P}_{\mathbb{N}}$.\\
\noindent \textbf{Case:} $S\cup \{\infty\}$, $S\subseteq
\mathbb{N}$, $1\in S$ and $\abs{S}\geq 2$: Write
$S=\{n_{k}:1=n_{1}<n_{2}<\cdots\}$.  Let $P_{n}$ and $P_{\infty}$ be
the subgroups of the multiplicative group of rational numbers as
before. Let $G_{n}$, $n\geq 1$, be the matrix group
\begin{align}
\nonumber
G_{n}=\left\{\left(\begin{matrix}1&x&y\\0&f&0\\0&0&g\end{matrix}\right):x,y\in
\mathbb{Q}, f\in P_{n},g\in P_{\infty}\right\}
\end{align}
and $H_{n}$ the subgroup consisting of the diagonal matrices in
$G_{n}$. Then as noted in Example 5.2 of \cite{MR2163938}, $G_{n}$ is amenable and
$L(G_{n})\cong \mathcal{R}$. It is also true that $L(H_{n})$ is a
singular masa in $L(G_{n})$ with \emph{Puk\'{a}nszky invariant}
$\{n,\infty\}$ $($see Prop. 2.5 \cite{MR2261688}$)$. Consider
$M_{n}=L(G_{n})\overline{\otimes}\mathcal{R}\cong \mathcal{R}$ and
consider the masa $A_{n}=L(H_{n})\overline{\otimes} A(1)$. Then
$A_{n}\subset M_{n}$ is a singular masa with
$Puk_{M_{n}}(A_{n})=\{1,n,\infty\}$ $($see Thm. 2.1 \cite{MR2163938}$)$. Fix $\alpha\in
\mathbb{P}_{S}$. Now consider
\begin{align}
\nonumber &\mathcal{M}_{\alpha}=\oplus_{n\in S}M_{n} \text{ and }
A_{\alpha}=\oplus_{n\in S}A_{n},\text{ where }
\tau_{\mathcal{M}_{\alpha}}(\cdot)=\sum_{n\in
S}\alpha_{n}\tau_{M_{n}}(\cdot),
\end{align}
where $\tau_{M_{n}}$ denotes the faithful normal tracial state of
$M_{n}$. Then
$\mathcal{M}_{\alpha}*L(\mathbb{Z})=\mathcal{M}_{\alpha}*L(\mathbb{F}_{1})\cong
L(\mathbb{F}_{2})$ \cite{MR1201693}, $A_{\alpha}$ is a singular masa
in $L(\mathbb{F}_{2})$ and
$Puk_{L(\mathbb{F}_{2})}(A_{\alpha})=S\cup\{\infty\}$ from Thm.
$3.2$ \cite{MR2261688}.\\
\indent There exist orthogonal projections $\{p_{n}\}_{n\in
S}\subset A_{\alpha}$ with the property that $\sum_{n\in S}p_{n}=1$ and
$\tau_{L(\mathbb{F}_{2})}(p_{n})=\alpha_{n}$, such that the
\emph{left-right-measure} of the inclusion $A_{\alpha}\subset
L(\mathbb{F}_{2})$ has $\lambda\otimes\lambda$ as a summand and
measures singular with respect to $\lambda\otimes\lambda$ on the
squares $p_{n}\times p_{n}$ $($here by abuse of notation we think of
$p_{n}$ as a measurable set which corresponds to the projection
$p_{n})$. The singular part on $p_{n}\times p_{n}$ has the property
that its $(\pi_{1},\lambda), (\pi_{2},\lambda)$ disintegrations are
non zero almost everywhere on $p_{n}$. Non conjugacy of $A_{\alpha}$
and $A_{\beta}$ for $\alpha\neq\beta$ follows easily from Lemma 3.6,
Thm. 3.8 of \cite{MR11932708}.\\

\noindent \textbf{Case:} $S\cup \{\infty\}$, $S\subseteq \mathbb{N}$,
$1\not\in S$ and $\abs{S}\geq 2$: Let $G_{n},H_{n}$ for $n\in
\mathbb{N}_{\infty}$ be the groups defined in Eq. \eqref{groups}.
Let $M_{n}=L(G_{n}\times G_{\infty})$ and $A_{n}=L(H_{n}\times
H_{\infty})$ for $n\in S$. Fix $\alpha\in \mathbb{P}_{S}$. Let
$\mathcal{M}_{\alpha,S}=\oplus_{n\in S}M_{n}$ and
$A_{\alpha,S}=\oplus_{n\in S}A_{n}$, where $\mathcal{M}_{\alpha,S}$
is equipped with the trace
$\tau_{\mathcal{M}_{\alpha,S}}(\cdot)=\sum_{n\in
S}\alpha_{n}\tau_{M_{n}}(\cdot)$, where $\tau_{M_{n}}$ denotes the
faithful normal tracial state of $M_{n}$. Replace the role of the
masa $A_{\alpha}$ in the previous case by $A_{\alpha,S}$. We omit
the details.\\
\noindent \textbf{Case:} $\{\infty\}$: Consider the hyperfinite
$\rm{II}_{1}$ factor $\mathcal{R}$ with a singular masa $A$ such
that $Puk_{\mathcal{R}}(A)=\{\infty\}$. Consider the inclusion
$B=\overline{\otimes}_{n=1}^{\infty}A\subset
\overline{\otimes}_{n=1}^{\infty}\mathcal{R}$. Since up to
isomorphism, there is one hyperfinite $\rm{II}_{1}$ factor, so
$B\subset \mathcal{R}$ is a masa from Tomita's theorem on commutants. Since $Puk(B)=\{\infty\}$ from Lemma 2.4 \cite{MR999998}, so $B$ is singular from Cor. 3.2 \cite{MR815434}. Clearly,
$\Gamma(B)=1$. The \emph{left-right-measure} of the inclusion
$B\subset \mathcal{R}$ is singular to the product class from Thm.
\ref{central_seq_th2}. Now apply Thm. \ref{unctblymany}.\\

\indent The above constructions lead to the following theorem.
\begin{thm}\label{arbitrary}
Let $S$ be an arbitrary $($could be empty$)$ subset of $\mathbb{N}$.
Let $k\in \left\{2,3,\cdots,\infty\right\}$ and let $\Gamma$ be any
countable discrete group. There exist uncountably many pairwise non conjugate
singular masas in $L(\mathbb{F}_{k}*\Gamma)$ whose Puk\'{a}nszky
invariant is $S\cup\{\infty\}$.
\end{thm}
\begin{proof}
We have already proved that there exist uncountably many pairwise non
conjugate singular masas $\{A_{\alpha}\}_{\alpha\in I}$, where $I$
is some indexing set, in $L(\mathbb{F}_{2})$ whose
\emph{Puk\'{a}nszky invariant} is $S\cup\{\infty\}$. One has
isomorphisms \cite{MR1201693}
\begin{align}
\nonumber & L(\mathbb{F}_{2})*L(\mathbb{F}_{k-2}*\Gamma)\cong L(\mathbb{F}_{k}*\Gamma) \text{ for }k\geq 2.
\end{align}
For $k\geq 2$, each $A_{\alpha}$ is a singular masa in
$L(\mathbb{F}_{k}*\Gamma)$ \cite{MR2261688}. Use Lemma 5.7, Prop.
5.10 \cite{MR2261688} to decide the non conjugacy of $A_{\alpha}$
and $A_{\beta}$ when $\alpha\neq \beta$, in the free product.
\end{proof}

\begin{thm}
There exist non conjugate singular masas $A,B$ in
$L(\mathbb{F}_{k})$, $2\leq k\leq \infty$, with same
measure-multiplicity invariant.
\end{thm}
\begin{proof}
Let $\mathcal{R}=L^{\infty}(\underset{{n\in
\mathbb{Z}}}\prod(\{0,1\},\mu))\rtimes \mathbb{Z}$, where
$\mu(\{0\})=\mu(\{1\})=\frac{1}{2}$ and the action is Bernoulli
shift. Then the copy of $\mathbb{Z}$ gives rise to a  masa $A\subset
\mathcal{R}$ whose multiplicity is $\{\infty\}$ and whose
\emph{left-right-measure} is the class of product measure. This
follows from the fact that the maximal spectral type of Bernoulli
action is Lebesgue measure and its multiplicity is infinite on the
orthocomplement of constant functions and Prop. 3.1 \cite{MR1940356}. Consequently, for $k\geq 2$,
$A\subset
\mathcal{R}*(\overset{k-1}{\underset{r=1}*}\mathcal{R})\cong
L(\mathbb{F}_{k})$ \cite{MR1201693} is a singular masa whose
\emph{left-right-measure} is the class of product measure and whose
multiplicity function is $m\equiv \infty$, off the diagonal. Let $B$
be any single generator masa of
$L(\mathbb{F}_{k})$. The same holds for the single
generator masas as well due to malnormality of group inclusions. $A$ is not conjugate to $B$, as the former
is not maximally injective, while the single generator
masas are maximally injective from Cor. 3.3 \cite{MR720738}.
\end{proof}

\begin{rem}
We end this paper with the following observation. The following example was constructed by Smith and Sinclair in Example 5.1 of \cite{MR2163938}. It produces an example of
a m.p. dynamical system which solves Banach's problem with the group under consideration being $\mathbb{Q}^{\times}$.
Consider the matrix groups $G=\left\{ (\begin{smallmatrix}f&x\\0&1\end{smallmatrix}):f\in \mathbb{Q}^{\times},x\in \mathbb{Q}\right\}$ and $H$ the subgroup of $G$ consisting of diagonal matrices. Then $L(G)$ is isomorphic to the hyperfinite
$\rm{II}_{1}$ factor, and, $L(H)$ is a singular masa in $L(G)$. Note that $H$ is a malnormal subgroup of $G$ and $G=N\rtimes H$,
where $N=\left\{ (\begin{smallmatrix}1&x\\0&1\end{smallmatrix}):x\in \mathbb{Q}\right\}$. Matrix multiplication shows that
$(\begin{smallmatrix}1&1\\0&1\end{smallmatrix})$ is a bicyclic vector of $L(H)$. The \emph{left-right-measure} of $L(H)\subset L(G)$ is the class of product measure $($see Example 6.2 \cite{MR2261688}$)$.
\end{rem}

$ $\\
\noindent \emph{\textbf{Acknowledgements}}: I thank Ken Dykema, my
advisor, and Roger Smith for many helpful discussions. I am also
indebted to Stuart White for valuable conversations on Tauer masas.

\end{document}